\documentclass[10pt]{amsart}

\usepackage{amssymb,amsthm,amsfonts,latexsym}
\usepackage{amsmath}
\usepackage{mathrsfs}
\usepackage{stmaryrd}
\usepackage{accents}
\usepackage[utf8]{inputenc}
\usepackage[]{hyperref}
\usepackage{color}
\usepackage{enumitem}
\usepackage{appendix}
\usepackage{tabularx}

\allowdisplaybreaks

\theoremstyle{plain}
\newtheorem{theorem}{Theorem}[section]
\newtheorem*{theorem*}{Theorem}
\newtheorem{lemma}[theorem]{Lemma}
\newtheorem{corollary}[theorem]{Corollary}
\newtheorem*{corollary*}{Corollary}
\newtheorem{proposition}[theorem]{Proposition}

\theoremstyle{definition}

\newtheorem{definition}[theorem]{Definition}
\newtheorem{hypothesis}[theorem]{Hypothesis}

\theoremstyle{remark}
\newtheorem{remark}[theorem]{Remark}

\theoremstyle{plain}
\newtheorem{innercustomgeneric}{\customgenericname}
\providecommand{\customgenericname}{}
\newcommand{\newcustomtheorem}[2]{%
  \newenvironment{#1}[1]
  {%
   \renewcommand\customgenericname{#2}%
   \renewcommand\theinnercustomgeneric{##1}%
   \innercustomgeneric
  }
  {\endinnercustomgeneric}
}

\newcustomtheorem{theoremNum}{Theorem}
\newcustomtheorem{lemmaNum}{Lemma}
\newcustomtheorem{corollaryNum}{Corollary}

\numberwithin{equation}{section}

 \def\QQ{\mathbb{Q}}
 \renewcommand\AA{{\mathbb{A}}}
 \renewcommand\SS{{\mathbb{S}}}
 \newcommand\ZZ{{\mathbb{Z}}}
 
 \def\A{\mathcal{A}}

 \def\F{{\mathcal F}}

 \def\K{{\mathcal K}}

 \def\N{{\mathcal N}}

 \def\S{{\mathcal S}}

 \def\Id{{\text{Id}}}
 
 \def\id{{\text{id}}}

 \def\PGU{{\text{PGU}}}
 
 \def\McL{{\text{McL}}}
 \def\Fi{{\text{Fi}}}

 \def\Aut{{\text{Aut}}}
 
 \def\Out{{\text{Out}}}
 \def\Inn{{\text{Inn}}}
 \def\Inndiag{{\text{Inn diag}}}
 \def\Outdiag{{\text{Out diag}}}

 \newcommand{\QC}{{\text{($\QQ$-QC)}}}
 \newcommand{\ZQC}{{($\ZZ$-QC)}}
 \newcommand{\OQC}{{(QC)}}

 \def\Sz{\text{Sz}} 
 \def\Ree{\text{Ree}}

\def\groupiso{\cong}
\def\homotequiv{\simeq}

\newcommand{\tq}{\mathrel{{\ensuremath{\: : \: }}}}

\newcommand\gen[1]{\left\langle#1\right\rangle}

      \makeatletter
      \def\@setcopyright{}
      \def\serieslogo@{}
      \makeatother
   
\usepackage{environ}

\newcounter{quote}
\renewcommand{\thequote}{(H\arabic{quote})}

\NewEnviron{myquote}{\vspace{1ex}\par
\refstepcounter{quote}%
\hfill\llap{\thequote}\hfill\parbox{\dimexpr \textwidth-2cm}%
{\normalsize{{\BODY}}}%
\hfill\vspace{1ex}\par}

\begin{document}

\title [An approach to Quillen's conjecture via centralizers of simple groups]{An approach to Quillen's conjecture via centralizers of simple groups}
   \author{Kevin Iv\'an Piterman}
   \address{Departamento de Matem\'atica \\IMAS-CONICET\\
 FCEyN, Universidad de Buenos Aires. Buenos Aires, Argentina.}
\email{kpiterman@dm.uba.ar}

\thanks{Supported by a CONICET postdoctoral fellowship and grants PIP 11220170100357, PICT 2017-2997, and UBACYT 20020160100081BA}

   \begin{abstract}
   We show that, for any given subgroup $H$ of a finite group $G$, the Quillen poset $\A_p(G)$ of nontrivial elementary abelian $p$-subgroups, is obtained from $\A_p(H)$ by attaching elements via their centralizers in $H$.
   We use this idea to study Quillen's conjecture, which asserts that if $\A_p(G)$ is contractible then $G$ has a nontrivial normal $p$-subgroup.
   We prove that the original conjecture is equivalent to the $\ZZ$-acyclic version of the conjecture (obtained by replacing contractible by $\ZZ$-acyclic).
   We also work with the $\QQ$-acyclic (strong) version of the conjecture, reducing its study to extensions of direct products of simple groups of order divisible by $p$ and $p$-rank at least $2$.
   This allows us to extend results of Aschbacher-Smith and to establish the strong conjecture for groups of $p$-rank at most $4$.
   \end{abstract}



\maketitle

\section{Introduction}

Given a finite group $G$ and a prime number $p$ dividing its order, let $\A_p(G)$ be the Quillen poset of nontrivial elementary abelian $p$-subgroups of $G$.
We can study the homotopical properties of $\A_p(G)$ by means of its order complex.
In \cite{Qui78}, Quillen proved that $\A_p(G)$ is contractible if $G$ has a nontrivial normal $p$-subgroup.
He also conjectured the converse giving rise to the well-known \textit{Quillen conjecture}.
That is, if $\A_p(G)$ is contractible then $G$ has a nontrivial normal $p$-subgroup.
Equivalently, if $G$ has no nontrivial normal $p$-subgroup, then $\A_p(G)$ is not contractible.
This conjecture has been widely studied during the past decades, but it remains open so far.

In this article, we consider the following versions of the conjecture.
Let $O_p(G)$ be the largest normal $p$-subgroup of $G$, and $\tilde{H}_*(X,R)$ the reduced homology of a finite poset $X$ (which is the homology of its order complex $\K(X)$), with coefficients in the ring $R$.

\vspace{0.2cm}

\begin{tabular}{cl}
\OQC & If $O_p(G) = 1$ then $\A_p(G)$ is not contractible.\\
\ZQC & If $O_p(G) = 1$ then $\tilde{H}_*(\A_p(G),\ZZ) \neq 0$.\\
\QC & If $O_p(G) = 1$ then $\tilde{H}_*(\A_p(G),\QQ) \neq 0$.
\end{tabular}

\vspace{0.2cm}


Note that \QC{} implies \ZQC, which implies the original conjecture \OQC.

The most important advances on the conjecture were achieved on the stronger version \QC.
Quillen established \QC{} for solvable groups, groups of $p$-rank at most $2$ and some families of groups of Lie type \cite{Qui78}. 
Later, various authors dealt with the $p$-solvable case (see \cite{Alperin,Diaz} and \cite[Ch.8]{Smi11}) and in \cite{AK90} Aschbacher and Kleidman showed \QC{} for almost simple groups.
In \cite{AS93}, Aschbacher and Smith proved $\QC$ for $p > 5$, under certain restrictions on the unitary components.
They strongly used the classification of finite simple groups.
Recently, in a collaboration work with Sadofschi Costa and Viruel \cite{PSV}, we proved new cases of the conjecture, not included in the previous results.
In \cite{PSV}, we worked with the integer version of the conjecture \ZQC{} and proved that it holds if $\K(\S_p(G))$ contains a $2$-dimensional and $G$-invariant subcomplex homotopy equivalent to itself.
Recall that $\S_p(G)$ is the Brown poset of nontrivial $p$-subgroups of $G$ and that $\A_p(G)\hookrightarrow \S_p(G)$ is a homotopy equivalence (see \cite[Proposition 2.1]{Qui78}).
In particular, the integer version holds for groups of $p$-rank at most $3$.
Recall that the $p$-rank of $G$ is the dimension of $\K(\A_p(G))$ plus one.

Further applications and results concerning the homotopy type of the $p$-subgroup complexes can be found in \cite{Brown,Gro1,Gro2,PW,Smi11}.
In \cite[(1.4)]{Moller}, the authors considered a version of the conjecture even stronger than the rational one \QC{}: if $O_p(G) = 1$ then the Euler characteristic of $\A_p(G)$ is not $1$.
We will not work with this version.

In this article, we approach the study of Quillen's conjecture via the examination of the centralizers of the elementary abelian $p$-subgroups on suitable subgroups.
Roughly, if $H$ is a nontrivial subgroup of $G$, we show that $\A_p(G)$ can be obtained first by passing from $\A_p(H)$ to a homotopy equivalent superset $\N(H)$, namely members $E\in\A_p(G)$ with $E\cap H\neq 1$, and then from $\N(H)$ to $\A_p(G)$ by attaching the remaining subgroups throughout their link in $\N(H)$.
If $E\in \A_p(G)$ and $E\cap H = 1$, its link in $\N(H)$ is $\A_p(C_H(E))$, where $C_H(E)$ is the centralizer of $E$ in $H$.
We can understand the homotopy type of $\A_p(G)$ from that of $\A_p(H)$ and the structure of these centralizers.
In some cases, we extract points $E\in \A_p(G)$ with contractible link in $\N(H)$, and this is guaranteed precisely when $O_p(C_H(E))\neq 1$.
In this way, we can work with smaller subposets and apply inductive arguments.
This approach has its roots in the previous work with E.G. Minian on the fundamental group of these complexes \cite{MP19}.
Some of these constructions were also considered by Segev and Webb \cite{Segev,SW}.
In our article, we put more emphasis on the attachment process and extraction of points, which seems to have been barely exploited.

We will study \ZQC{} and \QC{} by using the idea described above and working under the following \textit{inductive}
assumption.
Let $R = \ZZ$ or $\QQ$.

\begin{quote}
(H1)$_R$ \quad Proper subgroups and proper central quotients of $G$ satisfy ($R$-QC).
\end{quote}

By a proper central quotient of $G$ we mean a quotient of $G$ by a nontrivial central subgroup $Z\leq Z(G)$ (here $Z(G)$ denotes the center of $G$).
The hypothesis of the central quotients is motivated by the fact that if $Z\leq Z(G)$ is a $p'$-group, then $\A_p(G)$ is naturally isomorphic to $\A_p(G/Z)$ (see Proposition \ref{propCenterReduction}).

The inductive assumption (H1)$_R$ is valid in the context of a counterexample of minimal order to the conjecture.
That is, if $H$ satisfies ($R$-QC) for all $|H| < |G|$, then (H1)$_R$ holds for $G$.
Therefore, the reader may replace the content of (H1)$_R$ by this stronger inductive requirement.

Under certain extra hypothesis (H2) on $G$ which does not depend on the ring $R$, we will establish ($R$-QC) for $G$.
That is, we will show that for suitable (H2), (H1)$_R$ + (H2) implies ($R$-QC).
Theorem \ref{mainTheorem} is the first result with this format.
The motivation for this theorem comes from \cite[Proposition 1.6]{AS93}, which shows that $O_{p'}(G)=1$ if $G$ is a counterexample of minimal order to \QC{} and $p> 5$.
Recall that $O_{p'}(G)$ is the largest normal $p'$-subgroup of $G$.
Theorem \ref{mainTheorem} is intended to be a variation of \cite[Proposition 1.6]{AS93} designated to deal with the same reduction in a more general context, without further restrictions.

\begin{theoremNum}{1}\label{mainTheorem}
Let $G$ be a finite group and $p$ a prime number dividing its order.
Let $R = \ZZ$ or $\QQ$.
Suppose that (H1)$_R$ holds for $G$ and that:
\begin{quote}
(H2) \quad $O_{p'}(G) \neq 1$.
\end{quote}
Then $G$ satisfies ($R$-QC).

In particular, a counterexample of minimal order $G$ to (R-QC) has $O_{p'}(G) = 1$.
\end{theoremNum}

Our theorem has no restriction on the prime $p$ and it uses the classification of simple groups to a much lesser extent than the analogous result \cite[Proposition 1.6]{AS93}, which is stated for $p > 5$.
Moreover, our version is stated for both coefficient rings $\ZZ$ and $\QQ$, while the original result of Aschbacher-Smith is stated only for rational coefficients.
The proof of the integer version of Theorem \ref{mainTheorem},  requires an explicit description of a nontrivial cycle in the homology for the $p$-solvable case of the conjecture.
To that end, we employ the characterizations of D\'iaz Ramos \cite{Diaz}.

In combination with the results on the fundamental group \cite{MP19}, we conclude that the original conjecture \OQC{} and the integer homology version \ZQC{} are equivalent.

\begin{theoremNum}{2}\label{mainEquivalentConjectures}
The original Quillen's conjecture and the integer Quillen's conjecture \ZQC{} are equivalent.
That is, \OQC{} holds for all finite groups if and only if \ZQC{} holds for all finite groups.
\end{theoremNum}

We also show that certain components in the groups allows us to propagate homology to the Quillen poset, hence establishing the conjecture.
We prove the following theorem, which also has the format of (H1)$_R$ plus an extra hypothesis (H2), and concludes ($R$-QC).
We follow the notation of \cite{AS93,GL83} for the simple groups.

\begin{theoremNum}{3}\label{mainComponentsTheorem}
Let $R=\ZZ$ or $\QQ$.
Suppose that (H1)$_R$ holds for $G$ and that:
\begin{quote}
(H2) \quad $G$ has a component $L$ such that either $L\groupiso U_3(2^3)$ and $p = 3$, or else $L/Z(L)$ has $p$-rank $1$.
\end{quote}
Then $G$ satisfies ($R$-QC).
\end{theoremNum}

In the proof of this theorem, we study the structure of the centralizers of outer automorphisms of order $p$ in the simple component $L$ of the group.
Then we pass through the inflated subposet $\N(H)$ with $H = (LA)C_G(LA)$ for suitable $A\in\A_p(G)$ normalizing $L$, and propagate homology from $\N(H)$ to $\A_p(G)$.
On the other hand, Theorem \ref{mainComponentsTheorem} handles the groups containing a component isomorphic to $L_2(2^3)$ ($p = 3$), $U_3(2^3)$ ($p = 3$) or $\Sz(2^5)$ ($p = 5$), which were excluded by Aschbacher-Smith during  the analysis of the conjecture for odd $p$ (see Section \ref{sectionSmallCases} for a more detailed discussion).
Nevertheless, our theorem shows that their proof of Quillen's conjecture works in the same way without need of these restrictions on the components of the groups.
This allows us to extend the main result of Aschbacher-Smith to $p = 5$.
The extension of \cite[Main Theorem]{AS93} to $p = 3$ is not immediate since its proof depends on \cite[Theorem 5.3]{AS93}, which is stated for $p \geq 5$.

\begin{corollaryNum}{4}\label{mainExtensionToP5AS}
Theorem \cite[Main Theorem]{AS93} also holds for $p = 5$.
\end{corollaryNum}

We combine Theorem \ref{mainComponentsTheorem} and the results on the fundamental group \cite{MP19}, with the classification of the simple groups of low $p$-rank, the structure of their centralizers and the classification of groups with a strongly $p$-embedded subgroup (i.e. with disconnected Quillen's complex), to yield the $p$-rank $4$ case of the conjecture.

\begin{theoremNum}{5}\label{mainCoroPRank4}
If $G$ has $p$-rank at most $4$, then it satisfies \QC.
\end{theoremNum}

Finally, we deduce the following corollary in the context of a minimal counterexample to the conjecture, 
by the properties of the generalized Fitting subgroup and the automorphism group of a direct product of simple groups (see Remark \ref{remarkFittingDirectProduct}).
Let $R = \ZZ$ or $\QQ$.

\begin{corollaryNum}{6}\label{mainCorollary}
If $G$ is a minimal counterexample to ($R$-QC), then $G$ has $p$-rank at least $5$ and there exist simple groups $L_1,\ldots, L_r$ of $p$-rank at least $2$ and positive integers $n_1,\ldots, n_r$ such that \[L_1^{n_1}\times \ldots \times L_r^{n_r}\leq G\leq \prod_{i=1}^r \Aut(L_i) \wr \SS_{n_i}.\]
\end{corollaryNum}

Here, $\SS_{n}$ is the symmetric group on $n$ letters and $H\wr \SS_n$ denotes the wreath product of the natural permutation action of $\SS_n$ on the set $\{1,\ldots,n\}$.

\vspace{0.1cm}

\textbf{Acknowledgements.} I would like to thank El\'ias Gabriel Minian for his valuable suggestions that helped me to improve the presentation of this article.
I am also very grateful to Stephen D. Smith for his careful reading of this article and all of his generous and constructive comments which significantly improved the presentation.
In particular, for many suggestions on the format of Section \ref{sectionHomologyPropagation}, Theorem \ref{mainTheorem} and the consequences of these methods.

\section{Preliminary results}

In this section we establish the main definitions and tools that we will use throughout the paper.
We refer to \cite{AscFGT} for more details on finite group theory.

All the groups considered here are finite.
By a simple group we will mean a non-abelian simple group.
We adopt the conventions of \cite{GL83} for the names of the simple groups and their automorphisms.

Denote by $Z(G)$ the center of $G$.
Let $O_p(G)$ be the largest normal $p$-subgroup of $G$ and $O_{p'}(G)$ the largest normal $p'$-subgroup of $G$.
The Fitting subgroup $F(G)$ of $G$ is the largest normal nilpotent subgroup of $G$, and it is the direct product of the subgroups $O_p(G)$, for $p$ dividing the order of $G$.
For a fixed prime $p$, $\Omega_1(G) := \gen{x\in G:x^p = 1}$.
The $p$-rank of $G$ is
\[m_p(G) := 1 + \dim \K(\A_p(G)) = \max\{\log_p(|A|) \tq A\in \A_p(G)\cup \{1\}\}.\]

If $H,K\leq G$ are subgroups of $G$, then $N_H(K)$ denotes the normalizer of $K$ in $H$ and $C_H(K)$ the centralizer of $K$ in $H$.
Denote by $[H,K]$ the subgroup generated by the commutators between elements of $H$ and $K$.
If $g\in G$, write $H^g = g^{-1}Hg$.

A \textit{component} of $G$ is a subnormal quasisimple subgroup of $G$, i.e. a subgroup $L\leq G$ which is subnormal, perfect and $L/Z(L)$ is simple.
Two distinct components of $G$ commute.
The \textit{layer} of $G$, denoted by $E(G)$, is the (central) product of the components of $G$.
The \textit{generalized Fitting subgroup} of $G$ is $F^*(G) = F(G)E(G)$.
It can be shown that $[F(G),E(G)] = 1$ and that $F(G)\cap E(G) = Z(E(G))$, which is the product of the centres of the components of $G$.
The key property of this characteristic subgroup is that it is self-centralizing, i.e. $C_G(F^*(G)) = Z(F^*(G)) = Z(F(G))$.
An \textit{almost simple group} is a finite group $G$ such that $F^*(G)$ is a simple group.
Equivalently, $L\leq G\leq \Aut(L)$, where $L = F^* (G)$ is simple.
See \cite[Chapter 11]{AscFGT} for further details.

Recall that $\Out(H) = \Aut(H)/\Inn(H)$.
If $H$ is a group of Lie-type, $\Inndiag(H)$ denotes the subgroup of inner-diagonal automorphisms of $H$, and $\Outdiag(H) = \Inndiag(H)/\Inn(H)$.

\begin{remark}\label{remarkFittingDirectProduct}
If $O_p(G) = 1 = O_{p'}(G)$, then $F(G) \leq O_p(G)O_{p'}(G) = 1$, so $F^*(G) = E(G)$ and its center is trivial.
Therefore, $F^*(G) = L_1\ldots L_n$ is the direct product of the components $\{L_1,\ldots, L_n\}$ of $G$, which are (non-abelian) simple groups of order divisible by $P$.
By the self-centralizing property, $F^*(G)\leq G\leq \Aut(F^*(G))$.
Moreover, $\Aut(F^*(G))$ can be easily described by using the fact that if $L$ is a simple group, then $\Aut(L^n) \groupiso \Aut(L) \wr \SS_n$ and that $\Aut(L\times K) \groupiso \Aut(L)\times \Aut(K)$ if $L$ and $K$ are non-isomorphic simple groups.
\end{remark}

If $X$ is a finite poset, we can study its homotopy properties by means of its associated order complex $\K(X)$, whose simplices are the nonempty chains of $X$.
If $x\in X$ and $Y\subseteq X$ is a subposet, let $Y_{\geq x}  = \{y\in Y: y\geq x\}$.
Define analogously $Y_{>x}$, $Y_{\leq x}$ and $Y_{<x}$.
The link of $x$ in $Y$ is $Y_{<x}\cup Y_{>x}$.

Recall that if $f,g:X\to Y$ are two order preserving maps between finite posets $X$ and $Y$ such that $f\leq g$ (i.e. $f(x)\leq g(x)$ for all $x\in X$), then $f$ and $g$ are homotopic when regarded as simplicial maps.
Write $X\simeq Y$ if $\K(X)\simeq \K(Y)$.
Note that this is not the usual convention that we employed in the previous articles \cite{MP18,Pit19}.

We recall below Quillen's fiber lemma for finite posets.

\begin{proposition}[{\cite[Proposition 1.6]{Qui78}}]\label{propositionTheoremAPosets}
Let $f:X\to Y$ be an order preserving map between finite posets.
If $f^{-1}(Y_{\leq y})$ is contractible for all $y\in Y$ (resp. $f^{-1}(Y_{\geq y})$ is contractible for all $y\in Y$), then $f$ is a homotopy equivalence.
In particular, if $X \subseteq X_0$ and $X_{>x}$ is contractible for all $x\in X_0-X$ (resp. $X_{<x}$ is contractible for all $x\in X_0-X$) then $X\hookrightarrow X_0$ is a homotopy equivalence.
\end{proposition}

The \textit{Brown poset} $\S_p(G)$ is the poset of nontrivial $p$-subgroups of $G$.
The inclusion $\A_p(G)\hookrightarrow \S_p(G)$ is a homotopy equivalence by \cite[Proposition 2.1]{Qui78}, and if $O_p(G)\neq 1$ then $\A_p(G)$ is contractible (see \cite[Proposition 2.4]{Qui78}).

Quillen related the direct product of groups with the join of their $p$-subgroup posets.
The join of two posets $X*Y$ is the poset whose underlying set is the disjoint union of $X$ and $Y$, keeping the given ordering within $X$ and $Y$, and setting $x < y$ for each $x\in X$ and $y\in Y$.
Moreover, $\K(X*Y)$ equals the join of simplicial complexes $\K(X)*\K(Y)$, and this is homeomorphic to the topological join of $\K(X)$ and $\K(Y)$ (see \cite[Proposition 1.9]{Qui78}).
If $Y\subseteq X$ are finite posets and $x\in X$, then note that the link of $x$ in $Y$ is the join $Y_{<x} * Y_{>x}$.

\begin{proposition}[{\cite[Proposition 2.6]{Qui78}}]\label{propJoin}
$\A_p(G_1\times G_2)\simeq \A_p(G_1) * \A_p(G_2)$.
\end{proposition}

The following proposition shows that, in some sense, it is enough to study the homotopical properties of $\A_p(G)$ when $Z(G)$ is the trivial group.

\begin{proposition}\label{propCenterReduction}
Let $Z\leq Z(G)$.
The following hold.
\begin{enumerate}
\item If $Z$ is a nontrivial $p$-group, then $\A_p(G)$ is contractible.
\item If $Z$ is a $p'$-group then the induced map $\A_p(G) \to \A_p(G/Z)$ is an isomorphism of posets.
Moreover, $O_p(G/Z) \groupiso O_p(G)$.
\item In particular, if $G$ satisfies (H1)$_R$ and $Z\neq 1$ is a $p'$-group, then $G$ satisfies ($R$-QC), where $R=\ZZ$ or $\QQ$.
\end{enumerate}
\end{proposition}

\begin{proof}
Part (1) follows easily since $O_p(Z)\leq O_p(G)$.
Part (2) follows directly from the isomorphism theorems and Sylow's theorems.
For the ``Moreover" part of (2), note that if $H\leq G$ then $O_p(HZ/Z)\groupiso O_p(HZ)$ since $Z$ is a central $p'$-subgroup of $G$.

Finally, part (3) is a consequence of the definition of the (H1)$_R$ hypothesis and part (2).
\end{proof}

Below we give an immediate consequence of the (H1)$_R$ hypothesis.
Let $R = \ZZ$ or $\QQ$.

\begin{lemma}\label{lemmaCentralAndOmega1Reduction}
Assume that $G$ satisfies (H1)$_R$ and that
\begin{quote}
(H2) \quad $Z(G) \neq 1$ or $\Omega_1(G) < G$.
\end{quote}
Then $G$ satisfies ($R$-QC).
\end{lemma}

\begin{proof}
Suppose that $O_p(G) = 1$.
If $Z(G)\neq 1$, then by the above lemma $\A_p(G) = \A_p(G/Z(G))$.
Since $G/Z(G)$ is a proper central quotient of $G$, it satisfies ($R$-QC) by (H1)$_R$.
Moreover, we also have that $O_p(G/Z(G)) = O_p(G) = 1$, so $\tilde{H}_*(\A_p(G),R) = \tilde{H}_*(\A_p(G/Z(G)),R)\neq 0$.

If $\Omega_1(G) < G$, then $\Omega_1(G)$ satisfies ($R$-QC) by (H1)$_R$.
Note that $O_p(\Omega_1(G)) = 1$ since $\Omega_1(G)$ is normal in $G$.
Since $\A_p(G) = \A_p(\Omega_1(G))$, $\tilde{H}_*(\A_p(G),R) = \tilde{H}_*(\A_p(\Omega_1(G)),R) \neq 0$.
\end{proof}

In the next lemmas we recall some results that will play a fundamental role in the proof of our main theorems.
For a given subgroup $H\leq G$, we ``inflate" the subposet $\A_p(H)$ and then we show that the remaining points of $\A_p(G)$ are attached to this inflated subposet throughout their centralizers in $H$.

\begin{definition}
For $H\leq G$, let
\[\N(H):=\{E\in\A_p(G) \tq E\cap H \neq 1\}.\]
We sometimes abbreviate $\N_H = \N(H)$.
\end{definition}

We can also regard the poset $\N(H)$ as the ``neighbourhood" of $\A_p(H)$, and $\N(H)-\A_p(H)$ as the ``boundary" of this neighbourhood.
We give below some consequences of this definition. See also \cite{Segev,SW}.

\begin{lemma}\label{lemmaInflation}
If $H\leq G$ then $\A_p(H)\hookrightarrow\N(H)$ is a strong deformation retract.
\end{lemma}

\begin{proof}
Let $i:\A_p(H)\hookrightarrow\N(H)$ be the inclusion and $\varphi:\N(H)\to \A_p(H)$ the map defined by $\varphi(E) = E\cap H$.
Then $i$ and $\varphi$ are order preserving maps with $i\varphi \leq \Id_{\N(H)}$ and $\varphi i = \Id_{\A_p(H)}$.
\end{proof}

The following lemma shows that the elements outside $\N(H)$ attach to it via their centralizers in $H$.

\begin{lemma}\label{lemmaLinksInflation}
Let $H\leq G$ be a subgroup and let $E\in \A_p(G)$ be such that $E\cap H = 1$.
Then $\N(H)_{>E}$ is homotopy equivalent to $\A_p(C_H(E))$.
\end{lemma}

\begin{proof}
Let $f:\A_p(C_H(E)) \to \N(H)_{>E}$ and $g:\N(H)_{>E}\to \A_p(C_H(E))$ be the maps defined by $f(A) = AE$ and $g(A) = A\cap H$.
Then $fg(A) = (A\cap H)E\leq A$ and $gf(A) = (AE)\cap H = A$ (by modular law).
Hence $fg \leq \Id_{\N(H)_{>E}}$ and $gf = \Id_{\A_p(C_H(E))}$.
\end{proof}

The above lemma shows that the link in $\N(H)\simeq \A_p(H)$ of a point in $\A_p(G)-\N(H)$, which can be thought of as the ``upper" part of its attachment, is the Quillen poset of its centralizer in $H$.
The ``lower" link of a point $E\in\A_p(G)$ is the poset of nontrivial proper subspaces of $E$ (regarded as a vector space over $\mathbb{F}_p$, the finite field of $p$ elements), and it has the homotopy type of a nontrivial bouquet of $p^{m_p(E)(m_p(E)-1)/2}$ spheres of dimension $m_p(E)-2$.
We can rebuild $\A_p(G)$ from $\N(H)$ by attaching points in the following way.
Take a linear extension of the complement $\A_p(G) - \N(H) = \{E_1,\ldots, E_r\}$ such that $E_i\leq E_j$ implies $i\leq j$.
For each $0\leq i\leq r$, consider the subposet $X_i = \N(H) \cup \{E_1,\ldots, E_{i}\}$.
This gives rise to a filtration
\[\N(H) = X_0 \subseteq X_1\subseteq \ldots \subseteq X_r = \A_p(G),\]
where $X_i = X_{i-1}\cup \{E_i\}$ and the link of $E_i$ in $X_{i-1}$ is
$$(\A_p(E_i) - \{E_i\}) * \N(H)_{>E_i} \simeq \left(\bigvee_{l=1}^k \SS^{m_p(E_i)-2}\right) * \A_p(C_H(E_i)),$$ with $k = p^{m_p(E_i)(m_p(E_i)-1)/2}$.

This provides a useful way to understand the homotopy type of $\A_p(G)$ if we choose a convenient subgroup $H\leq G$ for which we know how these centralizers are.
In particular, if they are contractible, the homotopy type of $\A_p(H)$ does not change.

\begin{lemma}[{cf. \cite[Lemma 4.3]{PSV}}]\label{lemmaRetract}
Let $G$ be a finite group and let $H\leq G$.
In addition, suppose that $O_p(C_H(E))\neq 1$ for each $E\in\A_p(G)$ with $E \cap H = 1$.
Then $\A_p(G)\homotequiv \A_p(H)$.

\begin{proof}
Let $E\in \A_p(G) - \N(H)$.
By Lemma \ref{lemmaLinksInflation} $\N(H)_{>E} \simeq \A_p(C_H(E))$, which is contractible by hypothesis.
Finally, by Proposition \ref{propositionTheoremAPosets} and Lemma \ref{lemmaInflation}, $\A_p(G)\simeq \N(H)\simeq \A_p(H)$.
\end{proof}
\end{lemma}

For example, we can take $H$ to be $LC_G(L)$, where $L$ is a simple component of $G$.
Note that $F^*(G)\leq H$.
If $E\in\A_p(N_G(L))$ and $E\cap H = 1$ then $C_H(E) =C_L(E)C_G(LE)$ and $\A_p(C_H(E))\simeq \A_p(C_L(E)) * \A_p(C_G(LE))$.
The group $C_L(E)$ is the centralizer of an elementary abelian $p$-group acting on the simple group $L$, which can be described by using the classification of the finite simple groups.
We may also apply inductive arguments on $C_G(LE)$.

\section{The homology propagation lemma}\label{sectionHomologyPropagation}

The aim of this section is to propose a generalization of \cite[Lemma 0.27]{AS93}, stated in Lemma \ref{generalizedHomologyPropgation}.
Both lemmas allow to propagate nontrivial homology from proper subposets to the whole Quillen poset.
These tools will be very useful to establish Quillen's conjecture when we have an extra inductive assumption such as (H1)$_R$.

Our Lemma \ref{generalizedHomologyPropgation} shares the spirit of \cite[Lemma 0.27]{AS93} but with the extra feature that it can also be applied to proper subposets $X \subset \A_p(G)$.
This subposet $X$ will be typically chosen to be homotopy equivalent to $\A_p(G)$ but better behaved, in certain sense, than the Quillen poset.
In many cases, we will see that $X$ satisfies the hypotheses of Lemma \ref{generalizedHomologyPropgation} while $\A_p(G)$ does not.

Before proceeding with the proof of this lemma, we need some definitions and results of \cite{AS93}.
From now on, we suppress the coefficient notation on the homology and suppose that they are taken in the ring $R = \ZZ$ or $\QQ$.
The definitions given below do not depend on the coefficient ring.

If $X$ is a finite poset, denote by $\tilde{C}_*(X)$ its augmented chain complex with coefficients in $R$.
Recall that $\tilde{C}_n(X)$ is freely generated by the chains $(x_0 < x_1 < \ldots < x_n)$ in $X$.
Write $\tilde{Z}_n(X)$ for the subgroup of $n$-cycles and $\tilde{H}_*(X)$ for the reduced homology of $X$.
Denote by $X'$ the poset of nonempty chains of $X$.
Equivalently, $X'$ is the face poset of $\K(X)$.

\begin{definition}
Let $X$ be a finite poset.
A chain $a\in X'$ is \textit{full} if for every $x\in X$ such that $\{x\}\cup a$ is a chain we have that $x\in a$ or $x\geq \max a$.
A chain $b$ containing $a$ is called \textit{$a$-initial chain} if for every $x\in b-a$ we have that $x > \max a$.
\end{definition}

The following property was introduced by Aschbacher and Smith in \cite{AS93}.

\begin{definition}\label{definitionQDp}
We say that $G$ has the \textit{Quillen's dimension property at $p$}, $(QD)_p$ for short, if $\tilde{H}_{m_p(G)-1}(\A_p(G))\neq 0$.
That is, $\A_p(G)$ has nontrivial homology in the highest possible dimension.
\end{definition}

Observe that the top integer homology group of $\A_p(G)$ is always free, so this definition does not depend on the chosen coefficient ring $\ZZ$ or $\QQ$.
It is worth noting that finite groups may not satisfy this property in general.
This had been already observed by Quillen in \cite{Qui78}.

\begin{definition}\label{definitionExhibitsQDp}
Let $G$ be a finite group with $(QD)_p$ and let $m = m_p(G)-1$.
Take a nontrivial cycle $\alpha\in \tilde{H}_m(\A_p(G)) = \tilde{Z}_m(\A_p(G))$.
If the chain $a = (A_0 < A_1 < \ldots < A_m)$ is a addend of the cycle $\alpha$, we write $a\in \alpha$ and say that $a$ or $A_m$ \textit{exhibits} $(QD)_p$ for $G$.
Note that $a$ is a full chain.
\end{definition}

We state next a special configuration of the $p$-solvable case of the conjecture.
Its proof can be found in \cite[Theorem 8.2.12]{Smi11}.
See also \cite{Alperin,AS93,Diaz}.

\begin{theorem}\label{theoremPSolvableCaseQDp}
If $G = O_{p'}(G)A$, where $A$ is an elementary abelian $p$-group acting faithfully on $O_{p'}(G)$, then $G$ has $(QD)_p$ exhibited by $A$.
\end{theorem}

Next, we set up the proper context that we need to culminate with the proof of Lemma \ref{generalizedHomologyPropgation}.
We shall work under \cite[Hypothesis 0.15]{AS93}, which we state below.

\begin{hypothesis}[{Central product}]\label{hypothesisCP}
$H\leq G$ and $K\leq C_G(H)$ with $H\cap K$ a $p'$-group.
\end{hypothesis}

Hypothesis \ref{hypothesisCP} implies that $[H,K] = 1$ and $H\cap K\leq Z(H)\cap Z(K)$.
Moreover, we have that $$\A_p(HK) \simeq \A_p(H/H\cap K) * \A_p(K/H\cap K)$$ since $HK$ is a central product and the shared central subgroup $H \cap K$ is a $p'$-group (see \cite[Lemma 0.11]{AS93}).

Under appropriate circumstances, there will be nonzero cycles $\alpha$ and $\beta$ in the homology of $\A_p(H)$ and $\A_p(K)$ respectively, and they will give rise to a nontrivial cycle $\alpha \times \beta$ (the shuffle product) in the homology of $\A_p(HK)$.
The final goal is to show that this product cycle produces nontrivial homology in $\tilde{H}_*(\A_p(G))$.
In order to do that, we will ask for some subgroup $A\in \A_p(H)$ involved in $\alpha$ to satisfy suitable strong hypotheses.
With these hypotheses, if $\alpha\times\beta$ is the trivial cycle in $\tilde{H}_*(\A_p(G))$, we will reduce to a calculation in $\tilde{H}_*(\A_p(K))$ and then arrive to a contradiction.

The idea of this section is to perform the above homology computations in a typically proper subposet $X$, which, in general, will be constructed to be homotopy equivalent to $\A_p(G)$.
Hence, showing that $\alpha\times \beta$ is a nontrivial cycle in $\tilde{H}_*(X)$ will lead to nontrivial homology in $\A_p(G)$, as desired.

The reduction described above is in fact carried out inside a subposet lying over $\A_p(K)$, namely $\N(K)$.
Therefore, we shall take $X$ satisfying the following property.

\begin{definition}
Under Hypothesis \ref{hypothesisCP}, we call a subposet $X\subseteq \A_p(G)$ an $\N_K$-\textit{superset} if $\N(K)\subseteq X$.
\end{definition}

Moreover, the homotopy equivalence between $\N(K)$ and $\A_p(K)$ given in Lemma \ref{lemmaInflation} will allow us to translate calculations in the homology of $\N(K)$ into the homology of $\A_p(K)$.
This approach also allows us to assume weaker hypotheses on the subgroup $A\in \A_p(H)$ than those originally required in \cite[Lemma 0.27]{AS93}.
For example, in \cite[Lemma 0.27]{AS93} it is required that elements of $\A_p(G)_{>A}$ have the form $AB$ with $B\in \A_p(K)$.
We weaken this hypothesis in Lemma \ref{generalizedHomologyPropgation} by only asking that the smaller subset $X_{>A}$ be contained in $\N(K)$.

We proceed now to generalize the definitions and results coming after \cite[Hypothesis 0.15]{AS93}.

\begin{definition}
Assume Hypothesis \ref{hypothesisCP}.
Let $a = (A_0 < \ldots < A_m)$ be a chain of $\A_p(H)$ and $b = (B_0 < \ldots < B_n)$ be a chain of $\A_p(K)$.
We have the following chain in $\A_p(HK)$: $$a*b :=(A_0 < \ldots < A_m < B_0A_m <\ldots < B_nA_m).$$
\end{definition}

Let $c = (0,1,2,\ldots, m+n+1)$.
A \textit{shuffle} is a permutation $\sigma$ of the set $\{0,1,2, \ldots, m+n+1\}$ such that $\sigma(i) < \sigma(j)$ if $ i < j\leq m$ or $ m+1 \leq i < j$.
Let $\sigma(c) := (\sigma(0),\sigma(1),\ldots, \sigma(m+n+1))$.

With the notation of the above definition, let $C_j = A_j$ if $j\leq m$ or $B_{j-(m+1)}$ if $j\geq m+1$.
For a shuffle $\sigma$, define $(a\times b)_{\sigma}$ to be the chain whose $i$-th element is $C_{\sigma(0)}C_{\sigma(1)}\ldots C_{\sigma(i)}$.

\begin{definition}
[{\cite[Definition 0.21]{AS93}}]
Assume Hypothesis \ref{hypothesisCP}.
The shuffle product of $a$ and $b$ is $$a\times b := \sum_{\sigma \text{ shuffle}} (-1)^{\sigma} (a\times b)_{\sigma} \in \tilde{C}_{m+n+1}(\A_p(HK)).$$
Extend this product by linearity to $\tilde{C}_{*}(\A_p(H))$ and $\tilde{C}_{*}(\A_p(K))$.
\end{definition}

Recall that we are aiming to apply later Lemma \ref{generalizedHomologyPropgation} which, in contrast to \cite[Lemma 0.27]{AS93}, works with a potentially proper $\N_K$-superset $X$ of $\A_p(G)$.
In many situation it will be the case where $\A_p(H)\subseteq X$, hence $\tilde{C}_*(\A_p(H))\subseteq \tilde{C}_*(X)$ and the following lemmas are automatic.
However, our overall arguments do not require this assumption, so we supply the lemmas below to also cover these cases.

\begin{lemma}
\label{lemmaShuffleProductInNSubposet}
Assume Hypothesis \ref{hypothesisCP} and let $X$ be an $\N_K$-superset.
\begin{enumerate}[label=(\roman*)]
\item If $a\in X'\cap \A_p(H)'$, $b\in\A_p(K)'$ and $\sigma$ is a shuffle, then $(a\times b)_\sigma \in X'$.
\item If $\alpha \in \tilde{C}_*(X)\cap\tilde{C}_*(\A_p(H))$ and $\beta \in \tilde{C}_*(\A_p(K))$ then $\alpha\times \beta \in \tilde{C}_*(X)\cap \tilde{C}_*(\A_p(HK))$.
\end{enumerate}
\end{lemma}

\begin{proof}
If $C\in (a\times b)_\sigma$, then either $C\in a\subseteq X$ or else $C$ contains some subgroup $B\in b$.
In the later case, $C\cap K \geq B\neq 1$, so $C\in \N(K) \subseteq X$ since $X$ is an $\N_K$-superset.
This proves part (i).
Part (ii) follows from (i), by $\tilde{C}_*(X)\cap\tilde{C}_*(\A_p(H)) = \tilde{C}_*(X\cap \A_p(H))$ and a linearity argument.
\end{proof}

\begin{proposition}
[{cf. \cite[Corollary 0.23]{AS93}}]
\label{propositionShuffleProductCycles}
Under Hypothesis \ref{hypothesisCP}, if $\alpha\in \tilde{Z}_m(\A_p(H))$ and $\beta\in\tilde{Z}_n(\A_p(K))$ then $\alpha \times \beta\in \tilde{Z}_{m+n+1}(\A_p(HK))$.
In addition, if $X$ is an $\N_K$-superset and $\alpha\in \tilde{C}_*(X)$ then $\alpha\times\beta\in \tilde{Z}_{m+n+1}(X)$.
\end{proposition}

\begin{proof}
The first part is \cite[Corollary 0.23]{AS93}, and the second part follows from Lemma \ref{lemmaShuffleProductInNSubposet}.
\end{proof}

\begin{remark}
\label{remarkaInitialDecomposition}
Let $X$ be a finite poset and $a\in X'$.
Denote by $\tilde{C}_*(X)_a$ the subgroup of $a$-initial chains and by $\tilde{C}_*(X)_{\neg a}$ the subgroup of non-$a$-initial chains.
Clearly we have a decomposition $$\tilde{C}_*(X) = \tilde{C}_*(X)_a \bigoplus \tilde{C}_*(X)_{\neg a}.$$

Moreover, if $\partial$ denotes the boundary map of its chain complex, then $$\partial(\tilde{C}_*(X)_{\neg a}) \subseteq \tilde{C}_*(X)_{\neg a}.$$

If $\gamma\in \tilde{C}_*(X)$ then $\gamma = \gamma_a + \gamma_{\neg a}$, where $\gamma_a$ corresponds to the $a$-initial part of $\gamma$, and $$\partial\gamma = \partial(\gamma_a) + \partial(\gamma_{\neg a}) = (\partial(\gamma_a))_a + (\partial(\gamma_a))_{\neg a} + \partial(\gamma_{\neg a}).$$
\end{remark}

This remark yields the following lemma.

\begin{lemma}
[{cf. \cite[Lemma 0.24]{AS93}}]
\label{lemmaFullChainBorder}
If $a\in X'$ is a full chain then $(\partial\gamma)_a = (\partial\gamma_a)_a$.
\end{lemma}

The following lemma generalizes \cite[Lemma 0.25(i)]{AS93} to $\N_K$-supersets.

\begin{lemma}
[{cf. \cite[Lemma 0.25]{AS93}}]
\label{lemmaShuffleProductAndaInitial}
Assume Hypothesis \ref{hypothesisCP}.
\begin{enumerate}[label=(\roman*)]
\item If $a\in \A_p(H)'$ and $b\in \A_p(K)'$ then $(a\times b)_a = (a\times b)_{\sigma = \id} = a * b$ in $\tilde{C}_*(\A_p(HK))\subseteq \tilde{C}_*(\A_p(G))$;
\item In addition, if $X$ is an $\N_K$-superset and $a\in X'$, then the conclusion of (i) remains true in $\tilde{C}_*(X)$.
\end{enumerate}
\end{lemma}

\begin{proof}
Let $\sigma$ be a shuffle and $C\in (a\times b)_{\sigma}$ with $C\notin a$.
Then $C\geq B$ for some $B\in b$, so $C\in \N(K)$ (see the proof of Lemma \ref{lemmaShuffleProductInNSubposet}).
Since $H\cap K$ is a $p'$-group, $C\notin \A_p(H)$.
Therefore, $(a\times b)_{\sigma}$ is $a$-initial if and only if $\sigma = \id$, proving item (i).
Item (ii) follows from Lemma \ref{lemmaShuffleProductInNSubposet}.
\end{proof}

We prove now the mentioned generalization of \cite[Lemma 0.27]{AS93}.
The crucial point in our lemma is that it can be applied to $\N_K$-supersets, which in general will be homotopy equivalent to $\A_p(G)$.
Recall that we are working with coefficients in $R = \ZZ$ or $\QQ$.

\begin{lemma}
\label{generalizedHomologyPropgation}
Let $G$ be a finite group.
Let $H,K\leq G$ and $X\subseteq \A_p(G)$ be such that:
\begin{enumerate}[label=(\roman*)]
\item $H$ and $K$ satisfy Hypothesis \ref{hypothesisCP};
\item $X$ is an $\N_K$-superset;
\item There exist a chain $a\in\A_p(H)'\cap X'$ and a cycle $\alpha \in \tilde{C}_m(\A_p(H))\cap \tilde{C}_m(X)$ such that the coefficient of $a$ in $\alpha$ is invertible and $\alpha\neq 0$ in $\tilde{H}_m(\A_p(H))$ (for some $m\geq -1$);
\item In addition, such $a$ is a full chain in $X$ and $X_{>\max a}\subseteq \N(K)$;
\item $\tilde{H}_*(\A_p(K))\neq 0$.
\end{enumerate}
Then $\tilde{H}_*(X)\neq 0$.

In particular, under (i), hypotheses (ii), (iii) and (iv) hold if coefficients are taken in $\QQ$, $X =\A_p(G)$ and $H$ has $(QD)_p$ exhibited by $A\in \A_p(H)$ such that $\A_p(G)_{>A}\subseteq A\times K$ (this is the hypothesis in \cite[Lemma 0.27]{AS93}, so the present result is indeed a generalization).
\end{lemma}

We also mention that Lemma \ref{generalizedHomologyPropgation} does not require $H$ to have $(QD)_p$, which by contrast was fundamental in \cite[Lemma 0.27]{AS93}.
This assumption is relaxed within the statements of (iii) and (iv).

\begin{proof}[Proof of Lemma \ref{generalizedHomologyPropgation}]
We essentially carry out the original proof of \cite[Lemma 0.27]{AS93} inside $\tilde{C}_*(X)$ since $X$ is an $\N_K$-superset.

By hypothesis (v), there exists a cycle $\beta\in \tilde{C}_{n}(\A_p(K))$ which is not a boundary in $\tilde{C}_*(\A_p(K))$.
Choose a chain $a$ and a cycle $\alpha$ as in the hypothesis (iii).
Then $\alpha\times\beta \in \tilde{Z}_{m+n+1}(X)$ by hypotheses (i), (ii), (iii) and Proposition \ref{propositionShuffleProductCycles}.
We show that $\alpha\times\beta$ is a nontrivial cycle in the homology of $X$.

Suppose by way of contradiction that for some chain $\gamma\in \tilde{C}_{m+n+2}(X)$ we have that
\begin{equation}\label{equationBoundary}
\alpha\times\beta = \partial\gamma.
\end{equation}
Write $\beta = \sum_i q_i (B_0^i<\ldots <B^i_n)$ and $\gamma = \sum_{j\in J} p_j(C_0^j < \ldots < C_{m+n+2}^j)$.
Now take $a$-initial parts in both sides of the expression of (\ref{equationBoundary}).
\begin{equation}\label{equationainitial}
(\alpha\times\beta)_a = (\partial\gamma)_a.
\end{equation}
Note that no intermediate group lying in $X$ can be added within $a$ due to hypothesis (iv).
Let $A = \max a$.
By item (ii) of Lemma \ref{lemmaShuffleProductAndaInitial}, the left-hand-side of (\ref{equationainitial}) becomes
\begin{equation}\label{equationAlphaBeta}
(\alpha\times\beta)_{a} = q(a \times \beta) = q\sum_i q_i\,\, a\cup (AB_0^i < \ldots < AB_n^i),
\end{equation}
where by hypothesis (iii), $q\neq 0$ is the coefficient of $a$ in $\alpha$, and it is invertible.
The expression in (\ref{equationAlphaBeta}) is then equal to the right-hand-side of (\ref{equationainitial}), which using Lemma \ref{lemmaFullChainBorder} is
\begin{align}\label{equationGammaInitial}
(\partial \gamma)_{a} & = \sum_{j\in J'} p_j \sum_{k = m+1}^{m+n+2} (-1)^k \,\, a \cup (C_{m+1}^j<\ldots <\hat{C}_k^j <\ldots <C_{m+n+2})\nonumber \\
& = \sum_{j\in J'} p_j(-1)^{m+1} \sum_{k=0}^{n+1} (-1)^k a\cup (C_{m+1}^j<\ldots <\hat{C}_{k+m+1}^j <\ldots <C_{m+n+2}).
\end{align}
The hat notation $\hat{C}^j_k$ means that this term does not appear in the chain, and $J' = \{j\in J\tq a\subseteq (C_0^j < \ldots < C_{m+n+2}^j)\}$.
For $0 \leq k\leq n+1$, set $D_k^j := C_{k+m+1}^j$.
Since $C_{k+m+1}^j > A$, we have $D_k^j \in \N(K)$ by hypothesis (iv), so that $D_k^j\cap K \neq 1$.

We use now the $\sim$ operation, which maps an $a$-initial chain to its subchain beginning just after $\max a$.
Apply the $\sim$ operation on both sides of the equation of $a$-initial chains (\ref{equationainitial}) using the expressions of (\ref{equationAlphaBeta}) and (\ref{equationGammaInitial}) respectively.
The left-hand-side of (\ref{equationainitial}) becomes $q\tilde{\beta}$, where $$\tilde{\beta} = \sum_i q_i (AB_0^i < \ldots < AB_n^i) \in \tilde{C}_*(\A_p(\N(K))).$$
The right-hand-side of (\ref{equationainitial}) becomes $\partial\tilde{\gamma}$,  with $$\tilde{\gamma} = \sum_{j\in J'} p_j(-1)^{m+1}(D_0^j < \ldots < D^j_{n+1})\in \tilde{C}_*(\A_p(\N(K))).$$

Now we reduce the above homology computation in $\N(K)$ to a calculation in $\tilde{H}_*(\A_p(K))$.
Consider the homotopy equivalence given by the poset map $\varphi:\N(K)\to \A_p(K)$ of Lemma \ref{lemmaInflation}.
Denoting by $\varphi_*$ the induced chain map, we get the following equalities in $\tilde{C}_*(\A_p(K))$:
\begin{equation}\label{equationTorsion}
q\beta = \varphi_*(q\tilde{\beta}) = \varphi_*(\partial(\tilde{\gamma})) = \partial(\varphi_*(\tilde{\gamma})),
\end{equation}
where $\varphi_*(\tilde{\gamma})\in \tilde{C}_{n+1}(\A_p(K))$.
Since $q$ is invertible, we have found that $\beta$ is a boundary in the chain complex $\tilde{C}_*(\A_p(K))$, contradicting our initial assumption on $\beta$.
\end{proof}

\begin{remark}\label{remarkExtensionIntegers}
If coefficients are taken in $\QQ$, then the coefficient requirement in hypothesis (iii) is automatically guaranteed if $a\in\alpha$. 
If they are taken in $\ZZ$, then hypothesis (iii) implies that the coefficient of $a\in\alpha$ is $\pm 1$.
We may eliminate this restriction in hypothesis (iii) if we can take $\beta\in \tilde{H}_*(\A_p(K),\ZZ)$ of order  prime to the coefficient $q$ of $a$ in $\alpha$ by (\ref{equationTorsion}) (or if it is not a torsion element).
\end{remark}

\begin{remark}\label{remarkHypv}
Indeed, getting hypothesis (v) in the above lemma is in general the hard part.
In \cite{AS93}, this hypothesis is frequently obtained by applying \cite[Theorem 2.4]{AS93}, which has certain restrictions on the prime $p$.
One of our goals is to try to avoid restrictions on $p$, so we investigate for other methods to get this hypothesis.
\end{remark}

\section{The reduction \texorpdfstring{$O_{p'}(G) = 1$}{Op'(G) = 1}}\label{sectionNewCasesRationalHomology}

In this section we prove Theorem \ref{mainTheorem}, reducing the study of Quillen's conjecture to finite groups $G$ with $O_{p'}(G) = 1$.
Then we use this result to conclude that the original Quillen's conjecture \OQC{} is equivalent to the integer homology version \ZQC{} (see Theorem \ref{mainEquivalentConjectures}).

The idea of Theorem \ref{mainTheorem} is to establish a variant of \cite[Proposition 1.6]{AS93} but giving an easier proof by using Lemma \ref{generalizedHomologyPropgation}.
During the proof, we will construct a subposet $X$ of $\A_p(G)$ satisfying the hypotheses of Lemma \ref{generalizedHomologyPropgation}.
This route is comparatively elementary in contrast with \cite[Propositio 1.6]{AS93}, which quotes the strongly CFSG-dependent result \cite[Theorem 2.4]{AS93}.
Here, in our proof of Theorem \ref{mainTheorem} we will be able to avoid dependence on that result.
Instead, in the proof below, we will only need to quote Theorem \ref{theoremPSolvableCaseQDp} which depends less deeply on the CFSG.

\begin{proof}[Proof of Theorem 1]
Suppose that $O_p(G) = 1$.
Our goal is to show that $\tilde{H}_*(\A_p(G))\neq 0$, with coefficients in $R = \ZZ$ or $\QQ$.


First, note that if $H < G$ is a proper normal subgroup such that $\A_p(H)\simeq \A_p(G)$, then $O_p(H)\leq O_p(G) = 1$ and hence, by (H1)$_R$, $0\neq \tilde{H}_*(\A_p(H))\groupiso \tilde{H}_*(\A_p(G))$.
Therefore, we can suppose that no such subgroup exists:

\vspace{0.1cm}
(H3) If $H < G$ is a proper normal subgroup then $\A_p(H)\not\simeq \A_p(G)$.
\vspace{0.1cm}

Now we head to the construction of a homotopy equivalent subposet $X$ of $\A_p(G)$ to apply Lemma \ref{generalizedHomologyPropgation} and get our goal. 
To achieve this, we are going to deduce a series of properties on our group $G$ which will lead to the definition of $X$ and the choice of convenient subgroups $H$ and $K$ satisfying the hypotheses of Lemma \ref{generalizedHomologyPropgation}.
In view of Lemma \ref{lemmaCentralAndOmega1Reduction}, we can suppose that:

\vspace{0.1cm}
(1) $Z(G) = 1$ and $\Omega_1(G) = G$.
\vspace{0.1cm}

Let $L:=O_{p'}(G)$, which is nontrivial by (H2).
The claim below holds by (1) above.

\vspace{0.1cm}
(2) $C_G(L) < G$, and hence some $A\in \A_p(G)$ acts faithfully on $L$.
\vspace{0.1cm}

Recall that $A\in \A_p(G)$ acts faithfully on $L$ if and only if $C_A(L) = 1$, and $O_p(LA) = C_A(L)$.
Let 
\[\F = \{A\in\A_p(G) \tq A \text{ acts faithfully on }L\},\]
and
\[\N = \{A\in\A_p(G) \tq A \text{ acts non-faithfully on }L\}.\]
From this we deduce the following assertion.

\vspace{0.1cm}
(3) $\F$ and $\N$ are disjoint, $\A_p(G) = \F \cup \N$, $\F$ is nonempty by (2) and $\N = \N(C_G(L))$.
\vspace{0.1cm}

With an eye on the notation of Lemma \ref{generalizedHomologyPropgation}, if $A\in \F$, let $H_A:=LA$ and $K_A := C_G(LA) = C_G(H_A)$.
Note that $C_G(LA) = C_{C_G(L)}(A)$ and that $H_A\cap K_A \leq Z(H_A) = Z(LA)$ is a $p'$-group since $O_p(LA) = C_A(L) = 1$.
Therefore, we have that:

\vspace{0.1cm}
(4) $H_A$ and $K_A$ satisfy Hypothesis \ref{hypothesisCP}.
\vspace{0.1cm}

\vspace{0.1cm}
(5) If $A\in \F$, then $K_A = C_{C_G(L)}(A)$ and $\N_{>A} = \N(K_A)_{>A}$.
\vspace{0.1cm}

Since $C_G(LA) \leq C_G(L)$, we have that $\N(K_A) \subseteq \N(C_G(L)) = \N$.
This proves that $\N(K_A)_{>A}\subseteq \N_{>A}$.
For the other containment, if $B\in \N_{>A}$ then $B\leq C_G(A)$ and $C_B(L)\neq 1$.
Hence, $1\neq C_B(L) = C_G(A)\cap B \cap C_G(L) = B\cap C_G(LA) = B\cap K_A$.

Now we will see how the configuration of Lemma \ref{generalizedHomologyPropgation} brings new ideas beyond the analogous result of \cite{AS93}.
We show next how to get hypothesis (v) of this lemma (see Remark \ref{remarkHypv}).

If for all $A\in \F$ we have $O_p(C_{C_G(L)}(A)) = O_p(K_A) \neq 1$, by (3) and Lemma \ref{lemmaRetract}, then $\A_p(G) \simeq \A_p(C_G(L))$.
This contradicts (H3) since $C_G(L)$ is normal in $G$.
In consequence, we get:

\vspace{0.1cm}
(6) There is $A\in \F$ with $O_p(K_A) = 1$.
\vspace{0.1cm}

Now we are going to build the subposet $X$ by removing points of $\A_p(G)$ with contractible link, so that we preserve the homotopy type.
By (6), we can take $A\in \F$ of maximal $p$-rank subject to $O_p(K_A) = 1$.
Let $X = \A_p(G) - \F_{>A}$.

\vspace{0.1cm}
(7) If $B\in \A_p(G) - X$ then $X_{>B} = \N_{>B}$ is contractible.
In particular, $X\simeq \A_p(G)$.
\vspace{0.1cm}

If $B\in \A_p(G)-X = \F_{>A}$ then $X_{>B} = \N_{>B}$ since $\F_{>B}\subseteq \F_{>A}$ and $\A_p(G)_{>B} = \F_{>B} \cup \N_{>B}$.
It follows from Lemma \ref{lemmaLinksInflation} that $$X_{> B} = \N_{ > B} \simeq \A_p(C_{C_G(L)}(B)) = \A_p(K_{B}) \simeq *$$ since $O_p(K_B)\neq 1$.
Finally, by Proposition \ref{propositionTheoremAPosets}, $X\simeq \A_p(G)$.


\vspace{0.1cm}
(8) $\A_p(LA)\subseteq X$.
\vspace{0.1cm}

Let $B\in \A_p(LA)$.
Since $A$ is a Sylow $p$-subgroup of $LA$, there exists $g\in L$ such that $B\leq A^g$, so $C_B(L) \leq C_{A^g}(L) = (C_A(L))^g = 1$.
Therefore $B$ is faithful on $L$, that is, $B\in \F$, and $|B|\leq |A|$.
Hence $B\notin \F_{>A}$, which means that $B\in X$.

Now we check the hypotheses of Lemma \ref{generalizedHomologyPropgation} with $H = H_A = LA$ and $K = K_A = C_G(LA)$.
\begin{enumerate}[label=(\roman*)]
\item It holds by (4).
\item If $B\notin X$, then $B$ acts faithfully on $L$, so $1 = C_B(L) \geq C_B(LA) = B\cap K$.
In consequence, $\N(K)\subseteq X$ and $X$ is an $\N_K$-superset.
\item By Theorem \ref{theoremPSolvableCaseQDp} applied to $H = LA$, we can pick a nonzero element $\alpha\in \tilde{H}_m(\A_p(H))$, where $m = m_p(A)-1$.
Since $\tilde{Z}_m(\A_p(H)) = \tilde{H}_m(\A_p(H))$, $\alpha$ is actually a cycle, and by a dimension argument, it involves a full chain $a$.
Since $A$ is a Sylow $p$-subgroup of $H$, after conjugating $\alpha$, we may suppose that $A\in a$.
Moreover, by (8) $\tilde{C}_*(\A_p(H)) \subseteq \tilde{C}_*(X)$.\\
The coefficient of $a$ in $\alpha$ clearly is invertible if $R = \QQ$.
For $R = \ZZ$, this is also true, but it less immediate and depends on the results of \cite{Diaz}.
See below for further details.
\item By (iii) $a$ is a full chain, $A = \max a$ and $X_{>A} = \N(K_A)_{>A}$ by (3) and (5).
\item It holds by (H1)$_R$ since $O_p(K_A) = 1$ by the choice of $A$, and $K_A = C_G(LA) \leq C_G(L) < G$ by (2).
\end{enumerate}
By (7) and Lemma \ref{generalizedHomologyPropgation}, $\tilde{H}_*(\A_p(G))\groupiso \tilde{H}_*(X)\neq 0$.

We explain now how to obtain the invertible coefficient for $a$ in $\alpha$ if $R=\ZZ$, in order to fulfil hypothesis (iii) of Lemma \ref{generalizedHomologyPropgation}.
In the proof of (iii) above, we fix first $\alpha$, and then we choose $a\in\alpha$ a full chain.
Hence, it remains to show that some $a\in\alpha$ has coefficient equals to $\pm 1$ (see Remark \ref{remarkExtensionIntegers}).
This is possible by using the explicit description of a nontrivial cycle that D{\'i}az Ramos gave for the $p$-solvable case in \cite{Diaz}.
It follows from the proofs of \cite[Theorems 5.1, 5.3 \& 6.6]{Diaz}.

This concludes the proof of Theorem \ref{mainTheorem} for both version of the conjecture.
\end{proof}

Now we prove Theorem \ref{mainEquivalentConjectures}.
We recall first a result on the fundamental group of the $p$-subgroup posets and the almost simple case of the conjecture.

\begin{theorem}[{cf. \cite[Theorem 5.2]{MP19}}]\label{theoremFundamentalGroup}
If $G$ is not an almost simple group and $O_{p'}(G) = 1$, then $\pi_1(\A_p(G))$ is a free group.
\end{theorem}

\begin{theorem}[{\cite{AK90}}]\label{theoremAlmostSimpleCase}
If $G$ is an almost simple group, then $\tilde{H}_*(\A_p(G),\QQ)\neq 0$.
\end{theorem}

\begin{proof}[Proof of Theorem \ref{mainEquivalentConjectures}]
We only need to prove that if the original conjecture \OQC{} holds for all finite groups, then the integer homology version \ZQC{} holds.

Let $G$ be a group with $O_p(G) = 1$.
We shall prove that $\tilde{H}_*(\A_p(G),\ZZ)\neq 0$.
By induction, we can assume that \ZQC{} holds for every group $H$ with $|H|<|G|$, so $G$ satisfies (H1)$_\ZZ$.

If $O_{p'}(G)\neq 1$, then we are done by Theorem \ref{mainTheorem}.
Suppose that $O_{p'}(G) = 1$.
Further, by Theorem \ref{theoremAlmostSimpleCase}, we can suppose that $G$ is not an almost simple group.
In view of Theorem \ref{theoremFundamentalGroup}, we conclude that $\pi_1(\A_p(G))$ is a free group.
In this case, note that $H_1(\A_p(G),\ZZ) = 0$ if and only if $\pi_1(\A_p(G)) = 1$.
Since $G$ satisfies \OQC{}, some of its homotopy groups are nontrivial, so by the Hurewicz theorem $\tilde{H}_*(\A_p(G),\ZZ)\neq 0$.
\end{proof}

We use the rational version of Theorem \ref{mainTheorem} to extend some results of \cite{PSV} on the integer conjecture \ZQC{} to \QC.
Below we recall one of the main results of \cite{PSV}.

\begin{theorem}
[{\cite[Corollary 3.3]{PSV}}]\label{theoremQuillenDimension2}
Suppose that $\K(\S_p(G))$ admits a $2$-dimensional and $G$-invariant subcomplex homotopy equivalent to itself.
Then $G$ satisfies \ZQC.
\end{theorem}

\begin{corollary}
[{\cite[Corollary 3.4]{PSV}}]\label{coroPRank3}
The integer Quillen's conjecture \ZQC{} holds for groups of $p$-rank at most $3$.
\end{corollary}

\begin{corollary}\label{coroExtension2Dimensional}
Suppose that $G$ satisfies (H1)$_\QQ$ and that:
\begin{quote}
(H2) \quad $\K(\S_p(G))$ admits a $2$-dimensional and $G$-invariant subcomplex homotopy equivalent to itself.
\end{quote}
Then $G$ satisfies \QC.
\end{corollary}

\begin{proof}
Suppose that $O_p(G) = 1$.
We show that $\A_p(G)$ is not $\QQ$-acyclic.
By Lemma \ref{lemmaCentralAndOmega1Reduction} and Theorems \ref{mainTheorem} and \ref{theoremAlmostSimpleCase}, we may further assume that $Z(G) = 1$, $O_{p'}(G) = 1$ and that $G$ is not an almost simple group.

On the other hand, by Theorem \ref{theoremQuillenDimension2}, $\A_p(G)$ is not $\ZZ$-acyclic.
In order to prove that it is not $\QQ$-acyclic, we show that $\A_p(G)$ has free abelian homology.
Since it has the homotopy type of a $2$-dimensional complex, we only need to verify that $H_n(\A_p(G),\ZZ)$ is a free group for $n = 0,1,2$.

Clearly $H_2(K,\ZZ)$ and $H_0(K,\ZZ)$ are free abelian groups.
Finally, by Theorem \ref{theoremFundamentalGroup}, $\pi_1(\A_p(G))$ is a free group, so its abelianization $H_1(K,\ZZ)$ is a free abelian group.
This completes the proof.
\end{proof}

We can extend Corollary \ref{coroPRank3} of the $p$-rank $3$ case to the rational version \QC.

\begin{corollary}\label{coroStrongConjecturePRank3}
The rational Quillen's conjecture \QC{} holds for groups of $p$-rank at most $3$.
\end{corollary}

\section{Particular cases}\label{sectionSmallCases}

In this section we prove Theorem \ref{theoremExtensionOfExcludedCases}, which allows to eliminate components isomorphic to $L_2(2^3)$ ($p = 3$), $U_3(2^3)$ ($p = 3$) and $\Sz(2^5)$ ($p = 5$) for the study of Quillen's conjecture.
For that purpose, we will use the structure of the centralizers of outer automorphisms of order $p$ in these simple groups, which can be found in \cite{GL83}.
Recall that $R = \ZZ$ or $\QQ$.

\begin{theorem}\label{theoremExtensionOfExcludedCases}
Suppose that $G$ satisfies (H1)$_R$ and that:
\begin{quote}
(H2) \quad $G$ has a component $L$ such that $L/Z(L)$ is isomorphic to $L_2(2^3)$ ($p = 3$), $U_3(2^3)$ ($p = 3$) or $\Sz(2^5)$ ($p = 5$).
\end{quote}
Then $G$ satisfies ($R$-QC).
\end{theorem}

The groups with these components were excluded in the analysis made by Aschbacher and Smith \cite{AS93}.
We summarize next the scheme of the proof of \cite[Main Theorem]{AS93} to see why these cases were excluded.
Consider a counterexample of minimal order $G$ subject to failing \QC, with $p > 5$ and the restriction on the unitary components.
Note that in particular $G$ satisfies (H1)$_\QQ$.
The proof splits in three steps.

The first step consists on showing \cite[Proposition 1.6]{AS93}, i.e. that $O_{p'}(G) = 1$.
They proved this invoking \cite[Theorems 2.3 \& 2.4]{AS93}.
These theorems, stated for $p$ odd, require that $G$ does not contain components isomorphic to $L_2(2^3)$, $U_3(2^3)$ or $\Sz(2^5)$ with $p = 3,3,5$ respectively.
By Theorem \ref{mainTheorem}, we can get the same reduction over $G$ because of (H1)$_\QQ$, without further restrictions on $p$ and the components.

The second step of the proof is \cite[Proposition 1.7]{AS93}, which eliminates the components of $G$ for which all their extensions satisfy $(QD)_p$.
In this step, \cite[Theorems 2.3 \& 2.4]{AS93} are invoked again.
By Theorem \ref{theoremExtensionOfExcludedCases}, we can extend this step of their proof to $p\geq 3$.

Finally, the third step consists on deriving a contradiction by computing the Euler characteristic of the fixed point subposet $\S_p(G)^Q$ in two different ways, where $Q$ is a $2$-hyperelementary $p'$-subgroup of $G$ (see \cite[p.490]{AS93}).
Here, \cite[Theorem 5.3]{AS93} is applied, which is stated for $p\geq 5$.
Theorem \ref{theoremExtensionOfExcludedCases} enables to extend this step to $p = 5$ and, consequently, \cite[Main Theorem]{AS93} to $p\geq 5$, but not to $p\geq 3$.
This proves Corollary \ref{mainExtensionToP5AS}.
See also the comment below \cite[Main Theorem]{AS93} and the remark at \cite[p.493]{AS93}.

These kind of components $L$ are the obstruction to showing $(QD)_p$ in \cite[Section 3]{AS93} because of the subgroups of the form $E = \gen{f}\times A$, where $f$ is a field automorphism of $L$ and $A = O_p(C_L(f))\neq 1$.
Here, $A$ exhibit $(QD)_p$ for $L$ but it is not maximal-faithful in $LE > L$ since $E$ is above $A$.
Nevertheless, the special feature $O_p(C_L(f)) \neq 1$ is exactly the ingredient we want to construct a proper subposet $X\subset \A_p(G)$.
This key feature allows us to extract these kind of subgroups $E$ from the poset $\A_p(G)$ and apply Lemma \ref{generalizedHomologyPropgation} with $H = L$.

Before we proceed with the proof of Theorem \ref{theoremExtensionOfExcludedCases}, we recall some elementary and useful facts.

\begin{remark}\label{remarkQDpPropagationBySubgroups}
Suppose that $H\leq G$ and $m_p(H) = m_p(G) = m$.
Then we have an inclusion in the top dimensional homology group $\tilde{H}_{m-1}(\A_p(H))\subseteq \tilde{H}_{m-1}(\A_p(G))$.
In particular, if $H$ has $(QD)_p$ then so does $G$.

On the other hand, if $L = L_1\times \ldots \times L_n$ is a direct product and each $L_i$ has $(QD)_p$ then $L$ has $(QD)_p$.
This follows from the homotopy equivalence $\A_p(L) \simeq \A_p(L_1) * \ldots * \A_p(L_n)$ of Proposition \ref{propJoin} and the homology decomposition of a join.
\end{remark}

Denote by $\Inn(G)$ the subgroup of $\Aut(G)$ of inner automorphisms of $G$.
If $\phi\in\Aut(G) - \Inn(G)$, we say that $\phi$ \textit{induces an outer automorphism} on $G$.
If $L\leq G$ and $E \leq N_G(L)$, then we can describe the types of automorphisms (inner or outer) induced by the action of $E$ on $L$ via the map $E\to \Aut(L)$.

\begin{lemma}\label{lemmaNoIntersectionIsOuter}
Let $L\leq G$ and $E\leq N_G(L)$.
Then
\[E\cap (L C_G(L)) = \{x\in E \tq x \text{ induces an inner automorphism on }L\}.\]
In particular, $E\cap (L C_G(L)) = 1$ if and only if $E$ acts by outer automorphisms on $L$.
\end{lemma}

\begin{proof}
Clearly $E\cap (L C_G(L))$ acts by inner automorphisms on $L$.
If $x\in E$ induces an inner automorphism on $L$,
then there exists $y\in L$ such that $z=y^{-1}x$ acts trivially on $L$.
Therefore, $z\in C_G(L)$ and $x = yz\in LC_G(L)$.
\end{proof}

\begin{remark}\label{remarkPropertiesComponents}
Suppose that $C_G(F^*(G)) = 1$ and $F^*(G) = E(G)$.
In particular, $Z(E(G)) = 1$.
Let $L$ be a component of $G$.
If $B\in \A_p(G)$ is such that $B\cap L\neq 1$, then $B\leq N_G(L)$.
This holds because if $b\in B$ then $L^b\cap L\geq B\cap L$ is nontrivial, and it forces to $L^b = L$ (see \cite[(31.7)]{AscFGT}).

On the other hand, if $N = O_p(C_G(L))$ and $K$ is a component of $G$ distinct to $L$, then $K$ is a component of $C_G(L)$ and hence, $[N,K] = 1$.
In consequence, $N$ commutes with every component of $G$, that is, $1 = [N,E(G)] = [N,F^*(G)]$.
Since $C_G(F^*(G)) = 1$, $O_p(C_G(L)) = N = 1$.
\end{remark}

Now we prove Theorem \ref{theoremExtensionOfExcludedCases}.

\begin{proof}[Proof of Theorem \ref{theoremExtensionOfExcludedCases}]
Suppose that $O_p(G) = 1$.
We prove the rational version \QC, that is $\tilde{H}_*(\A_p(G),\QQ)\neq 0$.
The same proof works for the integer version \ZQC{} (see Remark \ref{remarkExtensionIntegersComponents}).

Let $L$ be a component of $G$ as in (H2).
By Lemma \ref{lemmaCentralAndOmega1Reduction} and Theorem \ref{mainTheorem} we can suppose that $Z(G) = 1$ and $O_{p'}(G) = 1$.
Thus, $L\groupiso L_2(2^3)$, $U_3(2^3)$ or $\Sz(2^5)$, with $p = 3$, $3$ or $5$ respectively.
We use the structure of the centralizers of the automorphisms of $L$.
We refer to (7-2), (9-1), (9-3) of \cite[Part I]{GL83} for further details.

\vspace{0.2cm}

\begin{itemize}
\setlength\itemsep{0.7em}
\item If $L = L_2(2^3)$ then $\Aut(L)\groupiso L\rtimes \gen{\phi}$, where $\phi$ induces a field automorphism of order $3$ on $L$, and $C_L(\phi)\groupiso L_2(2) \groupiso \SS_3 \groupiso C_3\rtimes C_2$.

Since $m_p(L) = 1$, $L$ has $(QD)_p$.

\item If $L = U_3(2^3)$ then $\Out(L) = \Outdiag(L)\rtimes C_6$ and $\Outdiag(L) \groupiso C_3$.
If $\phi\in \Aut(L)$ is a field automorphism of order $3$, $C_{\Inndiag(L)}(\phi) = C_L(\phi) \groupiso \PGU_3(2) \groupiso ((C_3\times C_3) \rtimes Q_8)\rtimes C_3$.
In particular, field automorphisms and diagonal automorphisms do not commute.

Since $m_p(L) = 2$ and $\A_p(L)$ is connected by Theorem \ref{disconnectedCasesTheorem}, $L$ has $(QD)_p$.

\item If $L = \Sz(2^5)$ then $\Aut(L) \groupiso L\rtimes \gen{\phi}$, where $\phi$ induces a field automorphism of order $5$ on $L$, and $C_L(\phi)\groupiso \Sz(2) \groupiso C_5\rtimes C_4$.

Since $m_p(L) = 1$, $L$ has $(QD)_p$.
\end{itemize}

\vspace{0.2cm}

In any case, $L$ has $(QD)_p$ and if $\phi\in N_G(L)$ induces a field automorphism on $L$ then $O_p(C_L(\phi))\neq 1$.
Let $\N = \{E\in\A_p(N_G(L)) \tq E\cap (LC_G(L))\neq 1\}$.
Note that $O_p(C_G(L)) = 1$ by Remark \ref{remarkPropertiesComponents}.

\vspace{0.2cm}

\textbf{Case 1:} $\A_p(N_G(L)) = \N$.
This case follows exactly as in Theorem \ref{theoremPRank1Components}, except if $L \groupiso U_3(2^3)$.
In that case, take a chain $a\in \A_p(L)'$ exhibiting $(QD)_p$ for $L$ and apply Lemma \ref{generalizedHomologyPropgation} with $H = L$, $K = C_G(L)$ and $X = \A_p(G)$.

\vspace{0.2cm}

\textbf{Case 2:} $\A_p(N_G(L))\neq \N$.
Here, every $E\in \A_p(N_G(L)) - \N$ induces outer automorphisms on $L$ and $|E| = p$ (see Lemma \ref{lemmaNoIntersectionIsOuter}).
Suppose that $E$ induces field automorphisms on $L$.
In particular $O_p(C_L(E))\neq 1$.
We show that $\A_p(G)_{>E}$ is contractible.

Let $Y = \{B\in \A_p(G) : N_B(L) > E\} \subseteq \A_p(G)_{>E}$ and let $M = C_L(E)C_G(LE)$.
Take $B \in \A_p(G)_{>E} - Y$.
We show first that $Y_{>B}\simeq \A_p(C_M(B))$.
Let $C\in Y_{>B}$.
We prove that $N_C(L)\cap M\neq 1$, and in particular $C\cap M\neq 1$.
If $N_C(L)$ contains inner automorphisms of $L$ or acts non-faithfully, then $N_C(L)\cap M\neq 1$ by Lemma \ref{lemmaNoIntersectionIsOuter}.
If $N_C(L)$ acts faithfully on $L$ and without inner automorphisms, then $N_C(L)$ embeds into both $\Aut(L)$ and $\Out(L)$, and it has $p$-rank at least $2$ since $E < N_C(L)$.
However, $N_C(L)$ can only contain field automorphisms of $L$ since diagonal and field automorphisms (of order $p$) do not commute (see (9-3) of \cite{GL83}).
On the other hand, a subgroup of $\Out(G)$ containing only field automorphisms of $L$ is cyclic of order $p$, a contradiction.
We have a well-defined homotopy equivalence $C \in{Y}_{>B}\mapsto C\cap M\in \A_p(C_M(B))$, with inverse $C \mapsto CB$.
Therefore, ${Y}_{>B}\simeq \A_p(C_M(B))$.

Now we prove that ${Y}_{>B}$ is contractible by showing that $O_p(C_M(B))\neq 1$.
Decompose $B = E\tilde{B}$, where $\tilde{B}$ is a complement to $E$ in $B$.
Since $N_B(L) = E$, $\tilde{B}$ acts regularly on the set $\{L^b:b\in \tilde{B}\}$.
Let $K = \langle L^b\tq b\in \tilde{B}\rangle$.
It is not hard to see that $C_K(B) \groupiso C_L(E)$.
Finally, observe that $C_K(B)$ is a normal subgroup of $C_M(B)$, so $O_p(C_M(B))\neq 1$.
Therefore, $Y_{>B} \simeq \A_p(C_M(B))$ is contractible.
By Proposition \ref{propositionTheoremAPosets}, $Y \hookrightarrow \A_p(G)_{>E}$ is a homotopy equivalence.

By taking $B = 1$ in the above reasoning, $Y \simeq \A_p(M)$ is contractible since $1\neq O_p(C_L(E))\leq O_p(C_L(E)C_G(LE)) = O_p(M)$.
In consequence, $\A_p(G)_{>E}\simeq Y$ is contractible.

In combination with Proposition \ref{propositionTheoremAPosets}, we have shown that the subposet
$$X_0 = \{ E\in \A_p(G) \tq \text{ if  } |E| = p \text{ then it does not induce field automorphisms on } L\}$$
is homotopy equivalent to $\A_p(G)$.
Now we extract the remaining elementary abelian $p$-subgroups acting faithfully on $L$ and containing field automorphisms.
Let
\[X = \A_p(G) - \{E\in\A_p(N_G(L))\tq C_E(L) = 1 \text{ and $E$ contains some field automorphism of $L$}\}.\]
If $F\in X_0 - X$ then it must contain inner automorphisms of $L$ by the same reasoning above.
That is, $1\neq F\cap (LC_G(L))$ by Lemma \ref{lemmaNoIntersectionIsOuter}, and, moreover, $|F:F\cap (LC_G(L))| = p$.
Hence, ${X}_{< F} =  \A_p(F\cap (LC_G(L))) \simeq *$.
By Proposition \ref{propositionTheoremAPosets}, $X \simeq X_0 \simeq \A_p(G)$.

To conclude the proof, apply Lemma \ref{generalizedHomologyPropgation} with the subposet $X$, $H = L$, $K = C_G(L)$, and $a\in\A_p(L)'$ any chain exhibiting $(QD)_p$ for $L$.
\end{proof}

\section{Components of \texorpdfstring{$p$}{p}-rank \texorpdfstring{$1$}{1}}

In this section we show that, under (H1)$_R$, if $G$ has a component of $p$-rank $1$ then $G$ satisfies ($R$-QC), with $R =\ZZ$ or $\QQ$ (see Theorem \ref{theoremPRank1Components} below).
This proves Theorem \ref{mainComponentsTheorem} when combined with Theorem \ref{theoremExtensionOfExcludedCases} and Remark \ref{remarkExtensionIntegersComponents}.
We refer to (7-13) of \cite[Part I]{GL83} for the main properties on simple groups of $p$-rank $1$.
Recall that there are no simple groups of $2$-rank $1$.

\begin{theorem}\label{theoremPRank1Components}
Suppose that $G$ satisfies (H1)$_R$ and that:
\begin{quote}
(H2)\quad $G$ has a component $L$ such that $L/Z(L)$ has $p$-rank $1$.
\end{quote}
Then $G$ satisfies ($R$-QC).
\end{theorem}

\begin{proof}
Similarly to Theorem \ref{theoremExtensionOfExcludedCases}, we can suppose that $Z(G) = 1$ and $O_p(G) = 1 = O_{p'}(G)$.
Hence $L$ is a simple group of $p$-rank $1$.
Recall that $p$ must be odd.
By Remark \ref{remarkPropertiesComponents}, $O_p(C_G(L)) = 1$.
Let $\N = \{E\in\A_p(N_G(L))\tq E\cap (L C_G(L))\neq 1\}$.
We split the proof in two cases.

\vspace{0.2cm}

\textbf{Case 1:} $\A_p(N_G(L))=\N$.
In this case, there are no outer automorphisms of order $p$ of $L$ inside $G$, and $\Omega_1(N_G(L)) \groupiso L\times \Omega_1(C_G(L))$.
Let $A\in \A_p(L)$.
If $B\in \A_p(G)_{>A}$, then $B\leq N_G(L)$ by Remark \ref{remarkPropertiesComponents} and hence, $B = AC_B(L)$.
Since $L$ has $p$-rank $1$, $A$ is a connected component of $\A_p(L)$ and it exhibits $(QD)_p$ for $L$.
The hypotheses of Lemma \ref{generalizedHomologyPropgation} are verified with $H = L$, $K = C_G(L)$, $a = (A)\in \alpha \in \tilde{C}_0(\A_p(L))$ and $X = \A_p(G)$.

\vspace{0.2cm}

\textbf{Case 2:} $\A_p(N_G(L))\neq \N$.
Here, every $E\in \A_p(N_G(L)) - \N$ induces outer automorphisms on $L$ by Lemma \ref{lemmaNoIntersectionIsOuter}, and has order $p$ since $m_p(L) = 1$ (see Table \ref{tableStronglypEmbedded}).
By Theorem \ref{theoremExtensionOfExcludedCases}, we can suppose that $L$ is not isomorphic to $L_2(2^3)$ ($p = 3$) nor to $\Sz(2^5)$ ($p = 5$).
Therefore, $m_p(LE) = 2$ and $LE$ has $(QD)_p$ (i.e. it is connected, see Table \ref{tableStronglypEmbedded}).
Let $K = C_G(LE)$ and suppose that $O_p(K) = 1$.
Take $A\in \A_p(LE)$ of order $p^2$ exhibiting $(QD)_p$ for $LE$, and observe that $LA = LE$ and $|A\cap L| = p$.
If $B\in \A_p(G)_{>A}$ then $B\cap L\neq 1$ by Remark \ref{remarkPropertiesComponents} and hence, $B \in \A_p(N_G(L))$.
Moreover, $C_B(L)\neq 1$ since $A\leq B/C_B(L)\leq \Aut(L)$, and the later has $p$-rank $2$.
In consequence, $B = AC_B(L)$.
Now apply Lemma \ref{generalizedHomologyPropgation} with $H = LE$, $K = C_G(LE)$, $A\in a\in \alpha\in \tilde{C}_1(\A_p(LA))$ exhibiting $(QD)_p$, and $X = \A_p(G)$.

If $O_p(C_G(LE))\neq 1$ for every $E \in \A_p(N_G(L)) - \N$,
consider the subposets $\F_0 = \A_p(N_G(L)) - \N$, $\F_1 = \{E\in \A_p(N_G(L)) \tq |E| = p^2, C_E(L) = 1, E\cap L\neq 1\}$ and $X = \A_p(G) - \F_1$.
We show that $X\simeq \A_p(G)$.
If $E\in \F_1$, then $E = (E\cap L)E_0$, where $E_0 \cap LC_G(L) = 1$.
Hence $E_0\in \F_0$ and $O_p(C_G(LE)) = O_p(C_G(LE_0)) \neq 1$.
Let $B\in {X}_{>E}$.
Since $m_p(\Aut(L)) = 2$ and $E$ acts faithfully on $L$ with $B\cap L\geq E\cap L\neq 1$, we have that $B\leq N_G(L)$, $C_B(L)\neq 1$ and $B = EC_B(L)$.
Therefore, ${X}_{>E}\simeq \A_p(C_G(LE))$, where the homotopy equivalence is given by $B\mapsto C_B(L)$ with inverse $C\mapsto CE$.
By Proposition \ref{propositionTheoremAPosets}, $X\simeq \A_p(G)$.
Finally, appeal to Lemma \ref{generalizedHomologyPropgation} on this subposet $X$, with $H = L$, $K = C_G(L)$ and $a = (A)\in \alpha \in \tilde{C}_0(\A_p(L))$, where $A\in \A_p(L)$.
This shows that $G$ satisfiers \QC.
For the integer version \ZQC, see Remark \ref{remarkExtensionIntegersComponents} below.
\end{proof}

\begin{remark}\label{remarkExtensionIntegersComponents}
In the proofs of Theorems \ref{theoremExtensionOfExcludedCases} and \ref{theoremPRank1Components}, we invoked Lemma \ref{generalizedHomologyPropgation} with some cycle $\alpha$ containing an arbitrary full chain $a$.
Note that we could have chosen first $\alpha$ and then $a\in \alpha$ in these proofs.
Since $\alpha$ is contained either in $\tilde{C}_1(X)$ or in $\tilde{C}_0(X)$, it can be taken to have coefficients equal to $\pm 1$.
For example, if $\alpha\in \tilde{C}_1(X)$, then pick $\alpha$ to be a simple cycle (that is, a cycle in the $1$-skeleton of the simplicial complex that does not self-intersect).
By Remark \ref{remarkExtensionIntegers}, Theorems \ref{theoremExtensionOfExcludedCases} and \ref{theoremPRank1Components} extend to the integer version of the conjecture \ZQC, and therefore, so does Theorem \ref{mainComponentsTheorem}.
\end{remark}

\section{The \texorpdfstring{$p$}{p}-rank \texorpdfstring{$4$}{4} case of Quillen's conjecture}\label{sectionsPRank4}

In this section we prove Theorem \ref{mainCoroPRank4}, establishing \QC{} for groups of $p$-rank at most $4$.
We use the results of the previous sections together with the classification of groups with a strongly $p$-embedded subgroup.
We provide in Appendix \ref{sectionAppendix} further details of this classification, as well as some properties of these groups.
We will see that the structure of the centralizers of the simple groups of low $p$-rank plays a fundamental role in the proof of this theorem.

The following elementary remark will be useful in the proof of Theorem \ref{mainCoroPRank4}.

\begin{remark}\label{remarkNonTrivialOuter}
Suppose that $L$ is a normal subgroup of $G$ such that $Z(L)$ is a $p'$-group.
If every order $p$ element of $G$ induces an inner automorphism on $L$, then $\Omega_1(G) \leq LC_G(L)$ by Lemma \ref{lemmaNoIntersectionIsOuter}. 
In particular, $\A_p(G)=\A_p(\Omega_1(G)) \simeq \A_p(L) * \A_p(C_G(L))$ by Proposition \ref{propJoin} and Lemma \ref{propCenterReduction}.
\end{remark}

\begin{proof}[Proof of Theorem \ref{mainCoroPRank4}]
Let $G$ of $p$-rank at most $4$ and suppose that $O_p(G) = 1$.
We prove that $\tilde{H}_*(\A_p(G),\QQ) \neq 0$.
Without loss of generality, we can assume that $G$ satisfies (H1)$_\QQ$, and hence that the following conditions holds:
\begin{enumerate}
\item $G = \Omega_1(G)$ and $Z(G) = 1$, by Lemma \ref{lemmaCentralAndOmega1Reduction};
\item $m_p(G) = 4$, by Corollary \ref{coroStrongConjecturePRank3}.
\item $O_{p'}(G) = 1$, by Theorem \ref{mainTheorem}.
\item $F^*(G) = L_1 \ldots L_n$ is the direct product of simple components $L_i$ of order divisible by $p$, by Remark \ref{remarkFittingDirectProduct}.
\item $G$ is not an almost simple group, by Theorem \ref{theoremAlmostSimpleCase}.
\item Every component of $G$ has $p$-rank at least $2$, by Theorem \ref{theoremPRank1Components}.
\end{enumerate}
Since $G$ is not almost simple, $n\geq 2$.
By (6), $m_p(L_i)\geq 2$ for all $i$, and since $4\geq m_p(F^*(G))\geq 2n$, we conclude that $n = 2$, $m_p(L_1) = 2 = m_p(L_2)$ and $m_p(F^*(G)) = 4$.

If both $\A_p(L_1)$ and $\A_p(L_2)$ are connected, then $L_1$ and $L_2$ have $(QD)_p$, and so does $F^*(G)$ and $G$ by Remark \ref{remarkQDpPropagationBySubgroups}.
Indeed, more generally, this argument shows the following case.

\vspace{0.2cm}

\textbf{Case 0:} if $G$ contains a direct product of distinct subgroups $G_1$, $G_2$ of $p$-rank $2$ for which $\A_p(G_1)$ and $\A_p(G_2)$ are connected with $O_p(G_i) = 1$, $i=1,2$, then $\tilde{H}_*(\A_p(G),\QQ)\neq 0$.

\vspace{0.2cm}

In consequence, we can suppose that $\A_p(L_1)$ is disconnected, i.e. that $L_1$ has a strongly $p$-embedded subgroup.
The possibilities for such $L_1$ are described in Theorem \ref{disconnectedCasesTheorem}.
In particular, if $p = 2$ then $L_1$ is isomorphic either to $L_2(2^2) \groupiso \AA_5$ or $U_3(2^2)$ by a $p$-rank argument (see Table \ref{tableStronglypEmbedded}).

On the other hand, by Remark \ref{remarkNonTrivialOuter}, if $G$ has a normal component $L_i$ then we can suppose that $G$ contains an outer automorphism of $L_i$ of order $p$, so $p\mid |\Out(L_i)|$.
Therefore, if $p$ is odd, both $L_i$ are normal in $G$ and it has to be that $p = 3$ and $L_1\groupiso L_3(2^2)$ by Table \ref{tableStronglypEmbedded}.

\vspace{0.2cm}

\textbf{Case 1:} $p = 2$ and $L_1 \groupiso \AA_5$ or $U_3(2^2)$.

\vspace{0.2cm}

If $f$ is an outer involutions of $\AA_5$ (resp. $U_3(2^2)$), then $C_{\AA_5}(f) \groupiso \SS_3$ of $2$-rank $1$, (resp. $C_{U_3(2^2)}(f)\groupiso\AA_5$ of $2$-rank $2$).
Both centralizers have disconnected Quillen's poset at $p=2$.
Moreover, $\Aut(\AA_5) =\SS_5 \groupiso \AA_5 \rtimes C_2$ and $\Aut(U_3(2^2)) = U_3(2^2)\rtimes C_4$, with $C_4$ inducing field automorphisms on $U_3(2^2)$.

We split the proof in two cases: when $L_1,L_2$ are permuted, and when they are normal in $G$.

\vspace{0.2cm}

\textbf{Case 1a:} some involution $x\in G$ permutes $L_1$ with $L_2$.

Then $N_G(L_1) = N_G(L_2)$ and it is a normal subgroup of $G$ of $p$-rank $4$ with $G = N_G(L_1)X$, where $X = \gen{x}$.
If $N_G(L_1)$ induces no outer automorphism of order $2$ on $L_1$, then $N_G(L_1) = L_1\times L_2$ and hence, $G\groupiso L_1\wr C_2$.
In this case, $\pi_1(\A_2(G))$ is a nontrivial free group by \cite[Theorem 5.6]{MP19} and therefore $\tilde{H}_1(\A_2(G),\QQ)\neq 0$.

Now assume that $N_G(L_1)$ induces some outer automorphism, say $f$, of order $2$ on $L_1$.
This eliminates $L_1\groupiso U_3(2^2)$ since $C_{L_1}(f)$ has $2$-rank $2$, which implies $m_2(N_G(L_1)) \geq 5$.
Hence $L_1\groupiso \AA_5$.
Now, $\AA_5\gen{f} \groupiso \SS_5$, which has $2$-rank $2$ and $(QD)_2$ (since $\A_2(\SS_5)$ is connected).
By Case 0 above, $N_G(L_1) = (L_1\times L_2) \gen{f}$ and $f$ induces an outer automorphism on both $L_1$ and $L_2$.
Then the subposet
\[\mathfrak{i}(\A_2(G)) := \{E\in \A_2(G) \tq E \text{ is the intersection of maximal elements of }\A_2(G)\}\]
has dimension $2$ (rather than dimension $3$ of $\A_2(G)$ itself).
This can be proved by using a similar argument to that of \cite[Examples 4.10 \& 4.11]{PSV}.
Finally, Corollary \ref{coroExtension2Dimensional} applies since $\K(\mathfrak{i}(\A_2(G)))$ is a $G$-invariant subcomplex of $\K(\S_2(G))$ and homotopy equivalent to $\K(\A_2(G))$.

\vspace{0.2cm}

\textbf{Case 1b:} $L_1$ is normal in $G$ (hence $L_2$ is also normal in $G$).

By Remark \ref{remarkNonTrivialOuter}, we may assume that:

\vspace{0.2cm}
\textbf{(H3)} Both components $L_1$ and $L_2$ admit nontrivial outer automorphisms from $G$.
\vspace{0.2cm}

Let $H = L_1 C_G(L_1)$, $\N := \N(H)$ and $\F:=\A_2(G) - \N$.
By (H3), $\F$ is non-empty, and by Proposition \ref{lemmaRetract}, we can suppose that some $E\in \F$ has $1 = O_2(C_H(E)) = O_2(C_{L_1}(E))O_2(C_G(L_1E))$.
Note that the elements of $\F$ have order $2$.

If $L_1\groupiso \AA_5$ and $E\in \F$, then $L_1E\groupiso \SS_5$, which has $(QD)_2$.
Fix $E\in \F$ with $O_2(C_G(L_1E)) = 1$ and take $A\in\A_2(L_1E)$ exhibiting $(QD)_2$ for $L_1E$.
Then $L_1E = L_1A$ and $O_2(C_G(L_1A)) = O_2(C_G(L_1E)) = 1$.
The hypotheses of Lemma \ref{generalizedHomologyPropgation} can be checked with $H = L_1E$, $K=C_G(L_1E)$ and $A$ exhibiting $(QD)_2$ for $L_1E$, so $\tilde{H}_*(\A_2(G),\QQ)\neq 0$.

Suppose now that $L_1\groupiso U_3(2^2)$ and $L_2\not\groupiso \AA_5$.
By Case 0 and (H3), we may assume that some involution $f\in G$ induces outer automorphisms on $L_1$ and $L_2$ simultaneously.
Since $C_{L_1}(f)$ has $2$-rank $2$, we conclude that $C_{L_2}(f)$ has $2$-rank $1$.
This forces to $L_2\groupiso L_2(q)$, with $q \geq 5$ odd and $f$ inducing diagonal automorphisms on $L_2$,
by the classification of simple groups of $2$-rank $2$ (see \cite[Theorem 48.1]{AscFGT}).
Moreover, if $\phi \in G$ is a field automorphisms of $L_2$ then $C_{L_2}(\phi)\groupiso L_2(q^{1/2})$ has $2$-rank $2$, which leads to $m_2((L_1L_2)\gen{\phi}) = 5$, a contradiction.
In conclusion, $G$ does not contain field automorphisms of $L_2$ and therefore, $G\leq \Aut(L_1) \times \Inndiag(L_2)$.

By Theorem \ref{disconnectedCasesTheorem}, $\A_2(L_2)$ is connected.
That is, $L_2$ has $(QD)_2$ exhibited by some $A\in\A_2(L_2)$ (of $2$-rank $2$).
Then $O_2(C_G(L_2A)) = O_2(C_G(L_2)) = 1$ by Remark \ref{remarkPropertiesComponents}, and if $B \in \A_2(G)_{>A}$ then $B/C_B(L_2)\leq \Inndiag(L_2)$, which has $2$-rank $2$.
Hence $C_B(L_2)\neq 1$ and $B = AC_B(L_2)$.
By Lemma \ref{generalizedHomologyPropgation} applied to $H = L_2$ and $K = C_G(L_2)$, we get $\tilde{H}_*(\A_2(G),\QQ)\neq 0$.

\vspace{0.2cm}

\textbf{Case 2:} $p = 3$ and $L_1\groupiso L_3(2^2)$.

Note that $\Out(L_3(2^2)) \groupiso D_{12} \groupiso C_3 \rtimes (C_2 \times C_2)$ and $\Inndiag(L_3(2^2)) \groupiso L_3(2^2)\rtimes C_3$, so without loss of generality $G\leq \Inndiag(L_3(2^2)) \times \Aut(L_2)$.
By Proposition \ref{propJoin} and the almost simple case of the conjecture, $G$ is not a direct product of almost simple groups.
Hence, there exists $C\in \A_p(G) - \A_p(L_1L_2)$ of order $3$ inducing diagonal automorphisms on $L_1\groupiso L_3(2^2)$.
Note that $L_1C\groupiso(L_1L_2)C / L_2 \groupiso \Inndiag(L_3(2^2))$.
If $C$ does not induce outer automorphisms on $L_2$, then $C\leq L_2C_G(L_2)$ and $G$ contains the normal subgroup $\Inndiag(L_3(2^2))$.
This implies that $C_G(L_2) = \Inndiag(L_3(2^2))$, so $G$ is the direct product of $\Inndiag(L_3(2^2))$ by some almost simple group $T\leq \Aut(L_2)$ with $F^*(T) = L_2$, a contradiction.
Therefore $C$ also induces outer automorphisms on $L_2$.

Recall that $C_{L_3(2^2)}(C)\groupiso \AA_5$ or $C_7\rtimes C_3$, both of $3$-rank $1$.
Since $\A_3(\Inndiag(L_3(2^2)))$ is connected (not simply connected) of dimension $1$ by Table \ref{tableStronglypEmbedded}, there exists $D\in\A_3(L_1C)$ of $3$-rank $2$ exhibiting $(QD)_3$ for $L_1C$.
Note that $L_1C = L_1D$.
If $O_3(C_G(L_1C)) = 1$, then $\tilde{H}_*(\A_3(G),\QQ)\neq 0$ by Lemma \ref{generalizedHomologyPropgation} applied with $X = \A_3(G)$, $H = L_1D$ and $K = C_G(L_1D) = C_G(L_1C)$.
On the other hand, if $O_3(C_G(L_1C))\neq 1$ for any choice of $C$, then let $H = L_1 C_G(L_1)$ and $\N =\N(H)$.
The subposet $\F:=\A_3(G) - \N$ consists of order $3$ subgroups acting by diagonal automorphisms on $L_1\groupiso L_3(2^2)$.
By Lemma \ref{lemmaRetract}, $\A_3(G)\simeq \A_3(H)$, so $\tilde{H}_*(\A_p(G),\QQ) = \tilde{H}_*(\A_p(H),\QQ)\neq 0$ by (H1)$_\QQ$.

This concludes the proof of the $p$-rank $4$ case.
\end{proof}

\appendix
\section{Groups with a strongly \texorpdfstring{$p$}{p}-embedded subgroup}\label{sectionAppendix}

In this appendix, we summarize some of the main results on the classification of the groups with a strongly $p$-embedded subgroup, so that the reader can consult them directly from here.
For further details see \cite{Asc93,GL83}.

Recall that a finite group $G$ has a \textit{strongly $p$-embedded subgroup} if there exists a proper subgroup $M < G$ such that $M$ contains a Sylow $p$-subgroup of $G$ and $M\cap M^g$ is a $p'$-group for all $g\in G-M$.
By \cite[Proposition 5.2]{Qui78}, $G$ has a strongly $p$-embedded subgroup if and only if $\A_p(G)$ is disconnected.
In the following theorem we state the classification of the groups with this property.

\begin{theorem}[{\cite[(6.1)]{Asc93}}]\label{disconnectedCasesTheorem}
The finite group $G$ has a strongly $p$-embedded subgroup (i.e. $\A_p(G)$ is disconnected) if and only if either $O_p(G) = 1$ and $m_p(G) = 1$, or $\Omega_1(G) / O_{p'}(\Omega_1(G))$ is one of the following groups:
\begin{enumerate}
\item Simple of Lie type of Lie rank $1$ and characteristic $p$,
\item $\AA_{2p}$ with $p \geq 5$,
\item $\Aut(L_2(2^3))$, $L_3(2^2)$ or $M_{11}$ with $p = 3$,
\item $\Aut(\Sz(2^5))$, ${}^2F_4(2)'$, $\McL$, or $\Fi_{22}$ with $p = 5$,
\item $J_4$ with $p = 11$.
\end{enumerate}
\end{theorem}

In Table \ref{tableStronglypEmbedded} we summarize some properties on the almost simple groups listed in Theorem \ref{disconnectedCasesTheorem}.
For more details on these assertions, see \textsection7, \textsection9 and \textsection10 of \cite[Part I]{GL83}.

\begin{table}[h]
\centering
\begin{tabular}{|c|c|c|c|}
\hline
Group $G$ & $\Out(G)$ & $m_p(G)$ & $m_p(\Out(G))$ \\ 
\hline 
\multicolumn{4}{|c|}{$p$-rank $1$ almost simple groups}\\
\hline
$G$ & cyclic Sylow $p$-subgroups & $1$ & $\leq 1$\\
\hline 
\multicolumn{4}{|c|}{Lie type of Lie rank $1$ in characteristic $p$}\\
\hline
$L_2(p^a)$ & $C_{\gcd(2,p^a-1)}\rtimes C_a$ & $a$ & $m_p(C_a)\leq 1$\\
\hline
$U_3(p^a)$ & $C_{\gcd(3,p^a+1)}\rtimes C_{2a}$ & $\begin{cases}a & p = 2\\
2a & p\neq 2
\end{cases}$ & $m_p(C_{2a})\leq 1$\\
\hline
$\Sz(2^a)$, $a\geq 3$ odd & $C_a$ & $a$ & $0$\\
\hline
$\Ree(3^a)$, $a\geq 3$ odd & $C_a$ & $2a$ & $m_p(C_{a})\leq 1$\\
\hline 
\multicolumn{4}{|c|}{Alternating groups, $p\geq 5$}\\
\hline
$\AA_{2p}$ & $C_2$ & $2$ & $0$ \\
\hline 
\multicolumn{4}{|c|}{$p = 3$ exceptions}\\
\hline
$\Aut(L_2(2^3))$ & $1$ & $2$ & $0$ \\
\hline
$L_3(2^2)$ & $D_{12}$ & $2$ & $1$ \\
\hline
$M_{11}$ & $1$ & $2$ & $0$\\
\hline 
\multicolumn{4}{|c|}{$p = 5$ exceptions}\\
\hline
$\Aut(\Sz(2^5))$ & $1$ & $2$ & $0$\\
\hline
${}^2F_4(2)'$ & $C_2$ & $2$ & $0$ \\
\hline
$\McL$ & $C_2$ & $2$ & $0$\\
\hline
$\Fi_{22}$ & $C_2$ & $2$ & $0$\\
\hline
\multicolumn{4}{|c|}{$p = 11$ exception}\\
\hline
$J_4$ & $1$ & $1$ & $0$\\
\hline
\end{tabular} 
\caption{Properties of almost simple groups with a strongly $p$-embedded subgroup.}
\label{tableStronglypEmbedded}
\end{table}


\begin{thebibliography}{11111}
\bibitem[Alp90]{Alperin} J.L. Alperin. \textit{A Lie approach to finite groups}, Groups---Canberra 1989, Lecture Notes in Math., vol. 1456, Springer, Berlin, 1990, pp. 1-9.

\bibitem[Asc93]{Asc93} M. Aschbacher. \textit{Simple connectivity of $p$-group complexes}, Israel J. Math. \textbf{82} (1993), no. 1-3, 1-43.

\bibitem[Asc00]{AscFGT} M. Aschbacher. \textit{Finite group theory}, second ed., Cambridge Studies in Advanced Mathematics, vol. 10, Cambridge University Press, Cambridge, 2000, 274 pages.

\bibitem[AK90]{AK90} M. Aschbacher and P.B. Kleidman. \textit{On a conjecture of Quillen and a lemma of Robinson}, Arch Math. (Basel) \textbf{55} (1990), no. 3, 209-217.

\bibitem[AS93]{AS93} M. Aschbacher and S.D. Smith. \textit{On Quillen's conjecture for the $p$-groups complex}, Ann. of Math. (2) \textbf{137} (1993), no. 3, 473-529.

\bibitem[Bro75]{Brown} K.S. Brown. \textit{Euler characteristics of groups: the $p$-fractional part}, Invent. Math. \textbf{29} (1975), no. 1, 1-5.

\bibitem[DR18]{Diaz} A. D\'iaz Ramos. \textit{On Quillen's conjecture for $p$-solvable groups}, J. Algebra \textbf{513} (2018), 246-264.

\bibitem[GL83]{GL83} D. Gorenstein and R. Lyons. \textit{The local structure of finite groups of characteristic $2$ type}, Mem. Amer. Math. Soc. \textbf{42} (1983), no. 276, vii+731.

\bibitem[Gro02]{Gro1} J. Grodal. \textit{Higher limits via subgroup complexes}, Ann. of Math. (2) \textbf{155} (2002), no. 2, 405-457.

\bibitem[Gro16]{Gro2} J. Grodal. \textit{Endotrivial modules for finite groups via homotopy theory}, arXiv e-prints, page arXiv:1608.00499 (2016).

\bibitem[JM12]{Moller} M.W. Jacobsen and J.M. M{\o}ller. \textit{Euler characteristics and M\"{o}bius algebras of $p$-subgroup categories}, J. Pure Appl. Algebra \textbf{216} (2012), no. 12, 2665-2696.

\bibitem[MP18]{MP18} E.G. Minian and K.I. Piterman. \textit{The homotopy types of the posets of $p$-subgroups of a finite group}, Adv. Math. \textbf{328} (2018), 1217-1233.

\bibitem[MP19]{MP19} E.G. Minian and K.I. Piterman. \textit{The fundamental group of the $p$-subgroup complex}, J. London Math. Soc. (2) (2020). In press.

\bibitem[Pit19]{Pit19} K.I. Piterman. \textit{A stronger reformulation of Webb's conjecture in terms of finite topological spaces}, J. Algebra \textbf{527} (2019), 280-305.

\bibitem[PSV19]{PSV} K.I. Piterman, I. Sadofschi Costa and A. Viruel. \textit{Acyclic $2$-dimensional complexes and Quillen's conjecture}, Publicacions Matem\`atiques (2019). In press.

\bibitem[PW00]{PW} J. Pulkus and V. Welker. \textit{On the homotopy type of the $p$-subgroup complex for finite solvable groups}, J. Austral. Math. Soc. Ser. A \textbf{69} (2000), no. 2, 212–228.

\bibitem[Qui78]{Qui78} D. Quillen. \textit{Homotopy properties of the poset of nontrivial $p$-subgroups of a group}, Adv. in Math. \textbf{28} (1978), no. 2, 101-128.

\bibitem[Seg96]{Segev} Y. Segev. \textit{Quillen's conjecture and the kernel on components}. Comm. Algebra \textbf{24} (1996), no. 3, 955-962.

\bibitem[SW94]{SW} Y. Segev, P. Webb. \textit{Extensions of G-posets and Quillen's complex}, Journal of The Australian Mathematical Society \textbf{57} (1994), 60-75.


\bibitem[Smi11]{Smi11} S.D. Smith. \textit{Subgroup complexes}, Mathematical Surveys and Monographs, vol. 179, American Mathematical Society, Providence, RI, 2011.

\end{thebibliography}


\end{document}